  \documentclass[12pt]{elsarticle}

\usepackage{amsmath}
\usepackage{amsthm}
\usepackage{amscd}
\usepackage{amssymb}
\usepackage{amsfonts}
\usepackage{graphics}
\usepackage{color}
\usepackage[USenglish]{babel}
\usepackage{titlesec}
\titlelabel{\thetitle \quad}

\date{}

\newlength{\defbaselineskip}
\long\def\salta#1{\relax}

\numberwithin{equation}{section}

\theoremstyle{plain}
\newtheorem{theorem}{Theorem}[section]
\newtheorem{proposition}[theorem]{Proposition}

\theoremstyle{definition}
\newtheorem{definition}[theorem]{Definition}

\newtheorem{remark}[theorem]{Remark}

\theoremstyle{remark}

\renewcommand{\theequation}{\thesection.\arabic{equation}}

\def\eps{\varepsilon}

\def\dys{\displaystyle}

\def\oeps{\Omega^\eps}

\def\t1p0{T^{1,p}_{0}(\Omega)}

\def\m2{M^{\frac{N(p-1)}{N-1}}(\Omega)}

\def\into{\int_{\Omega}}

\def\w-1p'{W^{-1,p'}(\Omega)}
\def\pw-1p'u{L^{p'}(0,1;W^{-1,p'}(\Omega))}

\def\dys{\displaystyle}

\def\lp'n{(L^{p'}(\Omega))^{N}}

\def\supp{\text{\rm{supp}}}

\def\oeps{\Omega^\eps}

\numberwithin{equation}{section}

\journal{Journal de Math\'ematiques Pures et Appliqu\'ees}

\begin{document}

\begin{frontmatter}



\title{
A semilinear elliptic equation \\ with a mild singularity at $u=0$:\\ existence and homogenization \\ 
------------------------------------ \\
{\it revised version, March 8, 2016} \\
{\it accepted for publication in J. Math. Pures et Appl.} \\
}


\author{Daniela Giachetti}

\address{Dipartimento di Scienze di Base e Applicate per l'Ingegneria \\ Facolt\`a di Ingegneria, Sapienza  Universit\`a di Roma \\ 
 Via Scarpa 16, 00161 Roma, Italy \\ 
 {\tt daniela.giachetti@sbai.uniroma1.it}}

\author{Pedro J. Mart\'inez-Aparicio}

\address{ Departamento de Matem{\'a}tica Aplicada y Estad\'istica \\ 
Universidad Polit{\'e}cnica de Cartagena \\ Paseo Alfonso XIII 52, 30202 Cartagena (Murcia), Spain \\
{\tt pedroj.martinez@upct.es}}

\author{Fran\c cois Murat\corref{corresponding author}}

\address{Laboratoire Jacques-Louis Lions, Universit\'e Pierre et Marie Curie \\  Bo\^ite Courrier 187, 75252 Paris Cedex 05, France \\
{\tt murat@ann.jussieu.fr}}

\begin{abstract}
In this paper we consider singular semilinear elliptic equations whose prototype is the following
\begin{equation*}
\begin{cases}
\displaystyle - div \,A(x) D u  = f(x)g(u)+l(x)& \mbox{in} \; \Omega,\\
u = 0 & \mbox{on} \; \partial \Omega,\\
\end{cases} 
\end{equation*}
where  $\Omega$ is an open bounded set of $\mathbb{R}^N,\, N\geq 1$, $A\in L^\infty(\Omega)^{N\times N}$ is a coercive matrix, $g:[0,+\infty[ \rightarrow [0,+\infty]$ is continuous, and $0\leq g(s)\leq {{1}\over{s^\gamma}}+1$ for every $s>0$, with $0<\gamma\leq 1$ and 
$f, l \in L^r(\Omega)$, $r={{2N}\over{N+2}}$ if  $N\geq 3$, $r>1$ if $N=2$, $r=1$ if $N=1$, $f(x), l(x)\geq 0$ a.e. $x \in \Omega$.

We prove the existence of at least one nonnegative solution as well as a stability result;  
we also prove  uniqueness if $g(s)$ is nonincreasing or ``almost nonincreasing".

Finally, we study the homogenization of these equations posed in a\break sequence of domains $\Omega^\eps$ obtained by removing many small holes from a fixed domain $\Omega$.

\bigskip

\noindent {\bf Titre fran\c cais} 

Une \'equation elliptique semi-lin\'eaire avec une singularit\'e faible en z\'ero : existence et homog\'en\'eisation 

\noindent {\bf R\'esum\'e fran\c cais}

Dans cet article nous \'etudions des \'equations elliptiques semi-lin\'eaires singuli\`eres dont le prototype est le suivant
\begin{equation*}
\begin{cases}
\displaystyle - div \,A(x) D u  = f(x)g(u)+l(x)& \mbox{dans} \; \Omega,\\
u = 0 & \mbox{sur} \; \partial \Omega,\\
\end{cases} 
\end{equation*}
o\`u  $\Omega$ est un ouvert born\'e de $\mathbb{R}^N,\, N\geq 1$, 
$A\in L^\infty(\Omega)^{N\times N}$ est une matrice coercive, 
$g:[0,+\infty[ \rightarrow [0,+\infty]$ est une fonction continue qui v\'erifie\break $0\leq g(s)\leq {{1}\over{s^\gamma}}+1$ pour tout $s>0$, 
avec $0<\gamma\leq 1$ et 
$f, l \in L^r(\Omega)$ avec $r={{2N}\over{N+2}}$ si  $N\geq 3$, $r>1$ si $N=2$ et $r=1$ si $N=1$, $f(x), l(x)\geq 0$\break p.p. $x \in \Omega$.

Nous d\'emontrons l'existence d'au moins une solution positive de ces \'equations et un r\'esultat de stabilit\'e ;  de plus nous d\'emontrons l'unicit\'e de la solution si $g(s)$ est d\'ecroissante ou  ``presque d\'ecroissante".

Nous \'etudions enfin l'homog\'en\'eisation d'une suite de ces \'equations pos\'ees dans des domaines 
$\Omega^\eps$ obtenus en perforant un domaine fixe $\Omega$ par des trous de plus en plus petits et de plus en plus nombreux.

\end{abstract}

\bigskip\bigskip

\begin{keyword}
Semilinear equations \sep singularity at $u = 0$ \sep existence \sep stability \sep uniqueness \sep homogenization

35B25 \sep 35B27 \sep 35J25 \sep 35J67


\end{keyword}

\end{frontmatter}

\section{Introduction}
We deal in this paper with nonnegative solutions to the following singular semilinear problem
\begin{equation}
\label{eqprima}
\begin{cases}
\displaystyle - div\, A(x) D u  = F(x,u) & \mbox{in} \; \Omega,\\
u = 0 & \mbox{on} \; \partial \Omega,\\
\end{cases} 
\end{equation}
where the model for the function $F(x,u)$ is 
\begin{equation*}
F(x,u)= f(x)g(u)+l(x),
\end{equation*}
for some continuous function $g(s)$ with $0\leq g(s)\leq {{1}\over{s^\gamma}}+1$ for every $s>0$, with $0<\gamma\leq 1$, and some nonnegative functions  $f(x)$ and $l(x)$ which belong to suitable Lebesgue spaces.\\
\indent 
Note that (except as far as uniqueness is concerned)  we do not require $g$ to be nonincreasing, so that  functions $g$ like 
%fm% equation centree
$$g(s)= {{1}\over{s^\gamma}}(2+ \sin {{1}\over{s}})$$ 
can be considered. 

In the present paper we are first interested in existence, uniqueness and stability results for this kind of problems. After this, we will study the asymptotic behaviour, as $\eps$ goes to zero, of a sequence of problems posed in domains $\Omega^{\eps}$ obtained by removing many small holes from a fixed domain $\Omega$, in the framework of \cite{CM}.

As far as existence and regularity results for this kind of problems are concerned, we refer to the classical paper \cite{BCR} by M.G.~Crandall, P.H.~Rabinowitz and L.~Tartar, and to the paper \cite{BO} by L.~Boccardo and L.~Orsina which inspired our work. We also refer to the references quoted in these papers as well as 
%fm%
those quoted in the paper \cite{BC} by L.~Boccardo and J.~Casado-D\'iaz
which deals with the homogenization of this problem for a sequence of\break matrices $A^\eps (x)$.\\
\indent In  \cite{BCR} the authors show the existence of a classical positive solution if the matrix $A(x)$, the boundary $\partial \Omega$ and the function $F(x, s)$ are smooth enough; the function $F(x, s)$, which is not supposed to be nonincreasing in $s$, is bounded from above uniformly for $x\in \overline \Omega$ and $s\geq 1$.
Boundary behaviour of $u(x)$ and $|D u (x)|$ when $x$ tends to $ \partial \Omega$ is also studied.   

In  \cite{BO} the authors study the problem \eqref{eqprima} with $F(x,u)=\frac{f(x)}{u^\gamma},\,\gamma>0$ and $f$ in Lebesgue spaces. They prove existence, uniqueness and regularity results depending  on the values of $\gamma$ and on the summability of $f$. Specifically, they prove the existence of strictly positive distributional solutions. In order to prove their results, they work by approximation and construct an increasing sequence $(u_n)_{n\in \mathbb{N}}$ of solutions to the (nonsingular) problems
\begin{equation*}
\begin{cases}
\displaystyle - div\, A(x) D u_n  = \frac{f_n(x)}{(u_n+\frac 1n)^\gamma} & \mbox{in} \; \Omega,\\
u_n = 0 & \mbox{on} \; \partial \Omega,\\
\end{cases} 
\end{equation*}
where $f_n(x)=\min\{f(x), n\}$.
This sequence satisfies, for every $\omega \subset \subset \Omega$,
\begin {equation} \label {strict}
u_n(x)\geq u_{n-1}(x)\geq...\geq u_1(x) \geq c_\omega >0,\,\, \forall x\in \omega.
\end {equation}
In order to prove this property, it is essential to assume that the nonlinearity $F(x,s)$ is nonincreasing in the $s$ variable and to use, as a main tool, the strong maximum principle. Note that \eqref {strict} provides the existence of a limit function $u=sup_n \,u_n$ which is strictly positive on every compact set $\omega$ of $\Omega$; in addition, \eqref{strict} implies that, on every such set $\omega$, the functions $\frac{f_n(x)}{(u_n+\frac 1n)^\gamma} $ are uniformly dominated by a   function $h_\omega \in L^1(\omega)$.  This allow the authors to prove that the function $u$ is a solution in the sense of distributions.

In the present paper,  we are interested in giving existence and stability results  without assuming that $F(x,s)$ is nonincreasing in the $s$ variable and without using the strong maximum principle in the proofs of these results. The main interest of this lies in the fact that this kind of proofs provides the tools to deal with the homogenization of the problem 
\begin{equation*}
\begin{cases}
\displaystyle - div \,A(x) D u^\eps  = F(x,u^\eps) & \mbox{in} \; \oeps,\\
u^\eps = 0 & \mbox{on} \; \partial \oeps,\\
\end{cases} 
\end{equation*}
when $\oeps$ is obtained by removing many small holes from $\Omega$
(see Theorem~\ref{homogenization}). Of course, the existence and stability results (see Theorem~\ref{EUS} and Theorem~\ref{est}) have also an autonomous interest, due to the more general assumptions and to a different method of proof.
 
Moreover, we point out that this method, which avoids using the strong maximum principle, also has a strong interest in other problems where one cannot expect the strict positivity of the solution on every compact set of $\Omega$.
Let us briefly describe some of these situations.

A first situation is the case of singular parabolic problems with $p\mbox{-laplacian}$ type principal part, $p>1$, and nonnegative data $u_0$ and $f$, whose model is the following
\begin{displaymath}
\left\{ \begin{array}{ll}
\dys u_t-div(|\nabla u|^{p-2}\nabla u)=f(x,t)\left(\frac{1}{u^\gamma}+1\right) & \textrm{in $\Omega\times(0,T)$},\\
u(x,t)=0 & \textrm{on $\partial\Omega\times(0,T)$},\\
u(x,0)=u_0(x) & \textrm{in $\Omega$},
\end{array} \right.
\end{displaymath}
where $\gamma>0$. In this case, due to the assumption $p>1$ and the fact that the initial datum $u_0$ is not assumed to be strictly positive, no method of expansion of positivity can be applied and one cannot guarantee that the solution is strictly positive inside $\Omega\times(0,T)$ (see \cite{BG}).

A second situation deals with existence and homogenization for elliptic singular problems in an open domain $Q$ of $\mathbb{R}^N$ which is made of an upper part $Q_1^\eps$ and a lower part  $Q_2^\eps$ separated by an oscillating interface $\Gamma^\eps$, when the boundary conditions at the interface $\Gamma^\eps$ are the continuity of the flux and the fact that this flux is proportional to the jump of the solution through the interface.
Our method also applies in this case.

A third situation where our method applies is the case of a singular semilinear problem which involves a zeroth-order term whose coefficient is a nonnegative measure  $\mu\in H^{-1}(\Omega)$, namely 
 \begin{equation}
\label{eqprimagen}
\begin{cases}
\dys u\geq 0 & \mbox{in }  \Omega,\\
\displaystyle - div\, A(x) D u+\mu u = F(x,u) & \mbox{in} \; \Omega,\\
u = 0 & \mbox{on} \; \partial \Omega.\\
\end{cases} 
\end{equation}
Problem \eqref{eqprimagen} naturally arises when performing the homogenization of \eqref{eqprima} (where there is no zeroth-order term) posed on a domain $\Omega^\eps$ obtained from $\Omega$ by perforating $\Omega$ by many small holes. Our method allows us to obtain results of existence, stability, uniqueness and homogenization, even if  the strong maximum principle does not hold true in general in such a context (see \cite{GMM1} and \cite{GMM2}).

\bigskip
 \indent In the present paper we consider the case $0<\gamma\leq 1.$ We  consider the case $\gamma>1$ (and more generally the case of a general singularity) in the papers \cite{GMM1} and  \cite{GMM2}
 %fm% nouveau
 (see also \cite{GMM3}).
Let us point out that in the latest case, no global energy estimate is available for the solutions when the singularity has a strong behaviour. This makes the problem more difficult, in particular from the point of view of homogenization. For this reason, we have to introduce a convenient (even if rather complicated) framework, in which we prove existence, stability, uniqueness and homogenization results. Let us emphasize that despite the changes which are made necessary by this framework, the method of the present paper provides the guide to follow also in the case of a general singularity.
  
The precise definition of the solution that we use in the present paper is given in Definition~\ref{sol}. Note that the solutions are nonnegative.

The keystone in our proofs is the analysis of the behaviour of the singular terms near the singularity, which is done in Proposition~\ref{prop3}  of Section 6.

 On the other hand, if we suppose that $F(x,s)$ is ``almost nonincreasing" in $s$) (see \eqref{eq0.2}), we  prove the uniqueness of the solution (see Theorem~\ref{uniqueness}).

Let us now come to the homogenization problem
\begin{equation*}
\begin{cases}
\displaystyle - div \,A(x) D u^\eps  = F(x,u^\eps) & \mbox{in} \; \oeps,\\
u^\eps = 0 & \mbox{on} \; \partial \oeps,\\
\end{cases} 
\end{equation*}
when $\oeps$ is obtained by removing many small holes from a given domain $\Omega$ according to the framework of \cite{CM}
%fm% nouveau
(for the study of this problem we have to assume that $N \geq 2$, see Remark~\ref{remN>1} below.)
 The general questions we are concerned with are the following: do the solutions $u^\eps$ converge to a limit $u$ when the parameter $\eps$ tends to zero? If this limit exists, can it be characterized? Will the result be the same result as in the non singular case?
In principle the answer is not obvious at all since, as $\eps$ tends to zero, the number of holes becomes greater and greater and the singular set for the right-hand side (which includes at least the holes' boundary) tends to "invade" the\break entire~$\Omega$. 

Actually we will prove that a strange term appears in the limit of the singular problem in the same way as in the non singular case studied in \cite{CM}. This result is a priori not obvious at all, and a very different behaviour could have been expected. 

\bigskip

We now describe the plan of the paper.
Section~$2$ deals with the precise assumptions on problem \eqref{eqprima}. In Section~$3$ we give the precise definition of a solution to problem \eqref{eqprima} which we will use in the whole of this paper. Section~$4$ is devoted to the statements of the existence, stability and uniqueness results; in addition a regularity result dealing with the boundedness of solutions is stated in this Section. In Section~$5$ we give the statement of the homogenization result in a domain with many small holes and Dirichlet boundary condition as well as a corrector result. In Section~$6$ we prove a priori estimates.  Section~$7$ is devoted to the proofs of the stability, existence and regularity results stated in Theorems~\ref{est} and \ref{EUS}  and in Proposition \ref{lem1}. In Section~$8$ we state and prove a comparison principle and we prove the uniqueness Theorem~\ref{uniqueness}. Finally we prove in Section~$9$ the homogenization Theorem~\ref{homogenization} and the corrector Theorem~\ref{corrector}.

\section{Assumptions}
In this Section, we give the assumptions on problem \eqref{eqprima}.

We assume that $\Omega$ is an open bounded set of $\mathbb{R}^N,\, N\geq 1$ (no regularity is assumed on the boundary $\partial\Omega$ of $\Omega$),  that the matrix $A$ satisfies 
\begin{equation}\label{eq0.0}
\begin{cases}
A \in L^\infty(\Omega)^{N\times N},\\
\exists \alpha>0,\, \,A(x)\geq \alpha I \,\,\,\,\,{\rm a.e.}\,\,  x\in\Omega,
\end{cases}
\end{equation}
and that the function $F$ satisfies
\begin{equation}
\label{car} 
\begin{cases}
F: \Omega\times [0, +\infty[ \rightarrow [0, +\infty] \,\, \text {is a Carath\'eodory function},\\ \mbox{i.e. } F \mbox{ satisfies}\\
i)\, \mbox{for a.e. } x\in\Omega, \, s\in [0,+\infty[\rightarrow F(x,s)\in [0,+\infty] \mbox{ is continuous},\\
ii)\, \forall s\in [0,+\infty[, x\in\Omega\rightarrow F(x,s)\in [0,+\infty] \mbox{ is measurable},
\end{cases}
\end{equation}

\begin{equation}
\label{eq0.1}
\begin{cases}
\exists \gamma, \exists h \mbox{ with }\\
i)\,0<\gamma\leq 1, \\
ii)\, h \in L^r(\Omega),\, r={{2N}\over{N+2}}\,\,\,\text {if}\,\, N\geq 3,\,\, r>1 \,\,\text{if} \,\, N=2,\,\,r=1\,\, \text {if} \,\,N=1,\\
iii)\, h(x)\geq 0 \, \, {\rm a.e.}\,\, x \in \Omega,\\
\mbox{such that}\\
iv)\, \displaystyle 0 \leq F(x,s)\leq h(x)\left({{1}\over{s^\gamma}}+1\right)\, \mbox{a.e. } x\in\Omega, \forall s>0.
\end{cases}
\end{equation}
\begin{remark}
The function $F=F(x,s)$ is a nonnegative Carath\'eodory function with values in $[0,+\infty]$. But, in view of (\ref{eq0.1} {\it iv}), the function $F(x,s)$ can take the value $+\infty$ only when $s=0$ (or, in other terms, $F(x,s)$ is always finite when $s>0$).
\qed
\end{remark}

On the other hand, for proving comparison and uniqueness results, we will assume that $F(x,s)$ is ``{\sl almost nonincreasing}" in $s$: denoting by $\lambda_1$ the first eigenvalue of the operator  $-div  \,{}^s\!A(x) D$ in $H_0^1(\Omega)$, where  $ {}^s\!A(x) =$ $= (A(x) +  {}^t\!A(x)) / 2$ is the symmetrized part of the matrix $A(x)$, we will assume that
\begin{equation}
\label{eq0.2}
\begin{cases}
\mbox{there exists } \lambda\mbox{ with } 0\leq \lambda<\lambda_1 \mbox{ such that}\\
F(x,s)-\lambda s\leq F(x,t)-\lambda t\,\,{\rm a.e.}\,\,\, x \in \Omega,\,\,\forall s,\forall t,\,0\leq t\leq s, 
\end{cases}
\end{equation}
or in other terms that $F (x, s) - \lambda s$ is nonincreasing in $s$ for some $\lambda$ such that $0\leq \lambda<\lambda_1$.

\begin{remark}
Note that \eqref{eq0.2} holds with $\lambda=0$ when $F$ is assumed to be nonincreasing. But if in place of \eqref{eq0.2} one only assumes that  the function
\begin{equation}
\label{unf}
s\in[0,+\infty]\rightarrow F(x,s)-\lambda_1 s\, \mbox{ is nonincreasing,}
\end{equation}
uniqueness of the solution to problem \eqref{eqprima} in general does not hold true, see Remark~\ref{rem82} below.
\qed
\end{remark}

\noindent {{\bf Notation}}

We denote by $\mathcal{D}(\Omega)$ the space of the  $C^\infty(\Omega)$ functions whose support is compact and included in $\Omega$, and by $\mathcal{D}'(\Omega)$ the space of distributions on $\Omega$.

Since $\Omega$ is bounded, $\|D w\|_{L^2(\Omega)^N}$ is a norm equivalent to $\|w\|_{H^1(\Omega)}$ on $H_0^1(\Omega)$. We set 
$$
\|w\|_{H_0^1(\Omega)}=\|D w\|_{L^2(\Omega)^N} ,\quad\forall w\in H_0^1(\Omega).
$$

For every $s\in \mathbb{R}$ and every $k>0$ we define  
$$s^+=\max\{s,0\},\,s^-=\max\{0,-s\},$$
$$T_k(s)=\max\{-k,\min\{s,k\}\},\quad G_k (s)= s- T_k (s).$$

For $l:\Omega\longrightarrow [0,+\infty]$ a measurable function we denote
$$
\{l=0\}=\{x\in\Omega: l(x)=0\},\,\,\{l>0\}=\{x\in\Omega: l(x)>0\}.
$$
\section{Definition of a solution to problem \eqref{eqprima}}
We now give a precise definition of a solution to problem \eqref{eqprima}.
\begin{definition}\label{sol}
Assume that the matrix $A$ and the function $F$ satisfy \eqref{eq0.0}, \eqref{car} and \eqref{eq0.1}. We will say that $u$ is a solution to problem \eqref{eqprima} if $u$ satisfies
\begin{equation}
\label{300}
u\in H^1_0(\Omega),
\end{equation}
\begin{equation}
\label{3000}
u\geq 0  \quad \rm{a.e.\, in}\,\, \Omega,
 \end{equation}
 \begin{equation}\label{L}
\begin{cases}
\forall  \varphi \in H_{0}^{1}(\Omega) \mbox{ with } \varphi \geq 0, \mbox{ one has } \\
\vspace{0.1cm}
\dys\int_\Omega F(x,u) \varphi<+\infty, \\
\dys\int_{\Omega} A(x)DuD\varphi= \int_\Omega F(x,u) \varphi.
\end{cases}
\end{equation}
\qed
\end{definition}

\begin{remark}
\label{32bis}
Given $\varphi \in H_0^1(\Omega)$, one can take $\varphi^+$ and $\varphi^-$ as test functions in  \eqref{L}. This implies that \eqref{L} is actually equivalent to 
\begin{equation}\label{Lbis}
\begin{cases}
\forall  \varphi \in H_{0}^{1}(\Omega), \mbox{ one has } \\ \vspace{0.1 cm}
\dys\int_\Omega F(x,u) | \varphi | <+\infty, \\
\dys\int_{\Omega} A(x)DuD\varphi= \int_\Omega F(x,u) \varphi.
\end{cases}
\end{equation}

This also proves that every solution $u$ to problem \eqref{eqprima} in the sense of Definition~\ref{sol} satisfies
$F(x,u)\varphi \in L^1(\Omega)$ for every $\varphi \in H_0^1(\Omega)$
and that  
$$F(x,u)\in L^1_{\mbox{{\tiny loc}}}(\Omega), \quad -div\, A(x)Du=F(x,u) \mbox{ in } \mathcal{D'}(\Omega).$$ 
\qed
\end{remark}

\begin{remark}
\label{38}
The nonnegative measurable function $F(x,u(x))$ can take infinite values when $u(x)=0$. The integral $\dys \into F(x,u)\varphi$ is therefore correctly defined as a number in $[0,+\infty]$ for every measurable function $\varphi\geq 0$. 

In \eqref{L} we require that this number is finite for every $\varphi\in H^1_0(\Omega)$, $\varphi\geq 0$, when $u$ is a solution to problem \eqref{eqprima} in the sense of Definition~\ref{sol}.  This in particular implies that 
\begin{equation}
\label{32bis2}
F(x,u(x))  \mbox{ is finite almost everywhere on } \Omega,
\end{equation}
or in other terms that 
\begin{equation}
\label{3100}
\mbox{meas}\{x\in\Omega: u(x)=0\mbox{ and }F(x,0)=+\infty\}=0.
\end{equation}

A result  which is stronger than \eqref{3100} will be given in Proposition~\ref{330}, and an even stronger result will be given in Proposition~\ref{331} and Remark~\ref{39}; note however that the strong maximum principle is used
to obtain the results of Proposition 3.5 and Remark 3.7. 
\qed
\end{remark}
\begin{proposition}
\label{330}
Assume that the matrix $A$ and the function $F$ satisfy \eqref{eq0.0}, \eqref{car} and \eqref{eq0.1}. Then every solution $u$ to problem \eqref{eqprima} in the sense of Definition~\ref{sol} satisfies 
\begin{equation}
\label{314}
F(x,0)=0 \quad  \mbox{for a.e.}\,\,x\in\{u=0\}
\end{equation}
 and 
 \begin{equation}
 \label{313bis}
\into F(x,u)\varphi=\int_{\{u>0\}} F(x,u)\varphi\quad \forall \varphi\in H^1_0(\Omega).
 \end{equation}
\end{proposition}
\begin{proof}
In Proposition~\ref{prop69} below we prove that for every $u$ solution to problem \eqref{eqprima} in the sense of Definition~\ref{sol} one has
\begin{equation}
\label{313}
\int_{\{u=0\}}F(x,u)\varphi=0\quad \forall \varphi\in H^1_0(\Omega),\varphi\geq 0,
\end{equation}
which of course implies \eqref{313bis}.

Writing  
$$
\{u=0\}=\Big(\!\{u=0\} \cap \{F(x,0)=0\}\!\Big)\cup \Big(\!\{u=0\} \cap \{0<F(x,0)\leq +\infty\}\!\Big)
$$
%%\begin{cases}
%\{u=0\}=\\=\Big(\{u=0\} \cap \{F(x,0)=0\}\Big)\cup\\\cup\Big(\{u=0\} \cap \{0<F(x,0)\leq +\infty\}\Big)
%\{u=0\}=\\=\Big(\{u=0\} \cap \{F(x,0)=0\}\Big)\cup \Big(\{u=0\} \cap \{0<F(x,0)\leq +\infty\}\Big)
%\end{cases}
%\end{equation*}
implies that \eqref{313} is equivalent to 
\begin{equation*}
\int_{\{u=0\}\cap\{0<F(x,0)\leq +\infty\}}F(x,u)\varphi=0\quad \forall  \varphi\in H^1_0(\Omega),\,\varphi\geq 0.
\end{equation*}
The latest assertion is equivalent to 
$$
\mbox{meas}\{x\in\Omega:u(x)=0\mbox{ and }0<F(x,0)\leq +\infty\}=0,
$$
which is equivalent to \eqref{314}. Proposition~\ref{330} is therefore proved.

Note that \eqref{314} is also equivalent to 
\begin{equation}
\label{316}
\begin{cases}
\{x\in\Omega:u(x)=0\}\subset\{x\in\Omega:F(x,0)=0\}\\
\mbox{except for a set of zero measure},
\end{cases}
\end{equation}
and also equivalent to
\begin{equation*}
\begin{cases}
\{x\in\Omega:0<F(x,0)\leq+\infty\}\subset\{x\in\Omega:u(x)>0\}\\
\mbox{except for a set of zero measure}.
\end{cases}
\end{equation*}
\end{proof}

The following Proposition~\ref{331} and Remark~\ref{39} assert that for every solution $u$ to problem \eqref{eqprima} in the sense of Definition~\ref{sol}, we can have two possibilities: either $u(x)>0$ a.e. in $\Omega$ or $u\equiv 0$ in $\Omega$. 

This assertion is stronger than \eqref{316}, but its proof uses the strong maximum principle. 

As pointed out in the Introduction, the strong maximum principle is one of the key tools used in the proofs of the results obtained in \cite{BO} by L.~Boccardo and L.~Orsina, results which inspired the present paper.

Note that, in contrast with the proofs of the results in \cite{BO}, the proofs of all the results in the present paper do not make use neither of the strong maximum principle nor of the results of Proposition~\ref{331} and Remark~\ref{39} below.

\begin{proposition}
\label{331}
Assume that the matrix $A$ and the function $F$ satisfy \eqref{eq0.0}, \eqref{car} and \eqref{eq0.1}. Then every solution $u$ to problem \eqref{eqprima} in the sense of Definition~\ref{sol} satisfies 
\begin{equation}
\label{3200}
\mbox{either } u\equiv0 \mbox{ or meas}\{x\in\Omega:u(x)=0\}=0.
\end{equation}
\end{proposition}

\begin{proof}
 We first recall the statement of the strong maximum principle Theorem~8.19 of \cite{GT}, or more exactly of its variant where $u$ is replaced by $-u$. In this variant, Theorem~8.19 of \cite{GT} becomes
\begin{align}
\label{am}
\begin{cases}
\mbox{Let } u\in W^{1,2}(\Omega) \mbox{ which satisfies } Lu\leq 0.\\
\mbox{If for some ball } B\subset\subset \Omega \mbox{ we have } \inf_{B}u=\inf_{\Omega}u\leq 0,\\\mbox{then } u \mbox{ is constant  in } \Omega.
\end{cases}
\end{align}
In our situation one has $Lu =div\, A(x)Du$, and $L u\leq 0$  is nothing but  $-div\, A(x)Du\geq 0$. Therefore \eqref{am} implies the following result
\begin{align}
\label{bm}
\begin{cases}
\mbox{Let } u\in H^1(\Omega) \mbox{ which satisfies } -div\, A(x)Du\geq 0 \mbox{ in } \mathcal{D}'(\Omega).\\
\mbox{If } u\geq 0 \mbox{ a.e. in } \Omega\mbox{ and if } \inf_{B}u=0 \mbox{ for some ball } B\subset\subset \Omega,\\
\mbox{then } u=0 \mbox{ in } \Omega,
\end{cases}
\end{align}
since when $u$ is a constant in $\Omega$ with $\inf_B u=0$, then $u=0$ in $\Omega$.

But one has the alternative:
\begin{align*}
\begin{cases}
\mbox{either } \inf_B u>0 \mbox{ for every ball } B\subset\subset \Omega,\\
\mbox{or there exists a ball } B\subset\subset \Omega \mbox{ such that }  \inf_{B} u=0.
\end{cases}
\end{align*}
In the first case, one has meas$\{x\in\Omega:u(x)=0\}=0$; in the second case, \eqref{bm} implies that $u\equiv0.$

This proves \eqref{3200}.
\end{proof}
\begin{remark}
Actually the proof of Proposition~\ref{331} (which uses the strong maximum principle) provides a result which is much stronger than \eqref{3200},
namely
\begin{equation}
\label{32003200}
\begin{cases}
\mbox{ either } u\equiv0, \\
\mbox{ or for every ball } B \subset\subset \Omega \mbox{ one has } \\ 
\,  \inf_{B}u \geq c(u, B) \mbox{ for some } c(u, B) \in \mathbb{R},  \, c(u, B) > 0.
\end{cases}
\end{equation}

Since the strong maximum principle continues to hold if the operator $-div\, A(x)Du$
is replaced by  $-div\, A(x)Du+a_0 u$, with $a_0~\in~L^\infty(\Omega)$, $a_0~\geq~0$,
both \eqref{3200} and \eqref{32003200} continue to hold for such an operator.

But when $a_0\geq 0$ does not belong to $L^\infty(\Omega)$ and is only a nonnegative element of $H^{-1}(\Omega)$
(this can be the case in the result of the homogenization process with many small holes that we will perform in Section~\ref{sechom}), 
the strong maximum principle does not hold anymore for the operator $-div\, A(x)Du\, +$ $+\, a_0 u$ (see \cite{GMM1} for a counter-example due to G. Dal Maso), 
and therefore \eqref{32003200} does not hold anymore for such an operator.
\qed
\end{remark}
\begin{remark}
\label{39}
If $u\equiv 0$ is a solution to problem \eqref{eqprima} in the sense of Definition~\ref{sol}, then Proposition~\ref{330} implies that $F(x,0) = 0$ for almost every $x \in \Omega$.

Conversely, if $F(x,0)\not\equiv 0$, $u\equiv 0$ is not a solution to problem \eqref{eqprima} in the sense of Definition~\ref{sol}, and Proposition~\ref{331} (or more exactly \eqref{32003200}) then implies that
$$
u(x)>0\mbox{ a.e. } x\in\Omega. 
$$
\qed
\end{remark}

\section{Statements of the existence, stability, uniqueness and regularity results}

In this Section we state results of existence, stability and uniqueness of the solution to problem \eqref{eqprima} in the sense of Definition~\ref{sol}. We also state a result (Proposition~\ref{lem1}) which provides the boundedness of the solutions under a regularity assumption on the datum $h$. 
\bigskip 

\begin{theorem}{\bf(Existence)}. 
  \label{EUS}
Assume that the matrix $A$ and the function $F$ satisfy \eqref{eq0.0}, \eqref{car} and \eqref{eq0.1}. Then there exists at least one solution $u$ to  problem \eqref{eqprima} in the sense of Definition~\ref{sol}. 
\end{theorem}
The proof of Theorem~\ref{EUS} will be done in Section~\ref{proofexistence}.
It is based on a stability result (see Theorem~\ref{est} below), and on a priori estimates of  $\|u\|_{H^1_0(\Omega)}$ and of $\dys \int_{\{u\leq \delta\}}F(x,u)v$ for every $v\in H_{0}^{1}(\Omega) $, $v\geq 0$ (see Propositions~\ref{lem2} and \ref{prop3} below);
these a priori estimates
 are satisfied by every solution $u$ to problem \eqref{eqprima} in the sense of Definition~\ref{sol}. 

  \begin{theorem}{\bf(Stability)}.
\label{est}
 Assume that the matrix $A$ satisfies assumption \eqref{eq0.0}. Let $F_n$ be a sequence of functions and $F_\infty$ be a function which satisfy assumptions
\eqref{car} and \eqref{eq0.1} for the same $\gamma$ and $h$. Assume moreover that 
\begin{equation}
 \label{num1}
 \mbox{ a.e. } x\in \Omega, \, F_n(x,s_n)\rightarrow\! F_\infty(x,s_\infty) \mbox{ if } s_n\rightarrow s_{\infty}, s_n\geq 0,s_\infty\geq0.
 \end{equation} 
Let $u_n$ be any solution to problem \eqref{eqprima}$_n$ in the sense of Definition~\ref{sol}, where \eqref{eqprima}$_n$ is the problem \eqref{eqprima} with $F(x,u)$ replaced by $F_n(x,u_n)$. 

Then there exists a subsequence, still labelled by $n$, and a function $u_\infty$, which is a solution to problem  \eqref{eqprima}$_\infty$ in the sense of Definition~\ref{sol}, such that
\begin{equation}
\label{num2}
%u_n\rightharpoonup u_{\infty} \, \mbox{ in } L^{\infty}(\Omega) \mbox{ weakly-star and a.e. in } \Omega,\\
 u_n\rightarrow  u_{\infty} \, \mbox{ in } H^1_0(\Omega)\, \mbox{strongly}.\\
\end{equation}
\end{theorem}

\vfill\eject

In the following Proposition we state a regularity result.

\begin{proposition}{\bf (Boundedness)}.
\label{lem1}
Assume that the matrix $A$ and the function $F$ satisfy \eqref{eq0.0}, \eqref{car} and \eqref{eq0.1}. Assume moreover that the function $h$ which appears in \eqref{eq0.1} satisfies
\begin{equation}
\label{nuch}
h \in L^t(\Omega),\, t>{{N}\over{2}}\,\,\,\text {if}\,\, N\geq 2,\,\,t=1\,\, \text {if} \,\,N=1.
\end{equation}

Then every $u$ solution to problem \eqref{eqprima} in the sense of Definition~\ref{sol} belongs to $L^\infty(\Omega)$ and satisfies the estimate
\begin{equation}
\label{num10}
\|u\|_{L^\infty(\Omega)}\leq 1+\frac{2}{\alpha} \, C(|\Omega|, N, t) \, \|h\|_{L^t(\Omega)},
\end{equation}
where the  constant $C(|\Omega, N, t)$  depends only on $|\Omega|$, $N$ and $t$ and  is nondecreasing in $|\Omega|$.
\end{proposition}
 
  Finally, our uniqueness result is a consequence of the comparison principle stated in Theorem~\ref{prop0} below. Note that these two results are the only results where the ``almost nonincreasing"  character in $s$ of the function $F(x,s)$ is used in the present paper.
 
 \begin{theorem}{\bf(Uniqueness)}. 
 \label{uniqueness}
 Assume that the matrix $A$ and the function $F$ satisfy \eqref{eq0.0}, \eqref{car} and \eqref{eq0.1}.  Assume moreover that the function $F$ also satisfies assumption \eqref{eq0.2}. Then the solution to problem \eqref{eqprima} in the sense of Definition~\ref{sol} is unique.
 
 \end{theorem}
 \begin{remark}
 When assumptions \eqref{eq0.0}, \eqref{car}, \eqref{eq0.1} as well as \eqref{eq0.2} hold true, Theorems~\ref{EUS}, \ref{est} and \ref{uniqueness} together  assert that problem \eqref{eqprima} is well posed in the sense of Hadamard in the framework of Definition~\ref{sol}.
 \qed
 \end{remark}

 \section{Statement of the homogenization result  in a domain with many small holes and Dirichlet boundary condition}
 \label{sechom}
 In this Section we deal with the asymptotic behaviour, as $\eps$ tends to zero, of nonnegative solutions to the singular semilinear problem
\begin{equation}
\tag{{\theequation}$^\eps$}
\label{eqprimah}
\begin{cases}
\displaystyle - div \,A(x) D u^\eps  = F(x,u^\eps) & \mbox{in} \; \oeps,\\
u^\eps = 0 & \mbox{on} \; \partial \oeps,\\
\end{cases} 
\end{equation}
where $u^\eps$ satisfies the homogenous Dirichlet boundary condition on the whole of the boundary of $\oeps$, when  $\Omega^\eps$ is a perforated domain obtained by  removing many small holes from a given open bounded set  $\Omega$ in $\mathbb{R}^N, N \geq 2$, with a repartition of those many small holes producing a ``strange term" when $\eps$ tends to $0$.

We begin by describing in Subsection~\ref{51} the geometry of the perforated domains and the framework introduced in \cite{CM} (see also \cite{DG} and \cite{MH}) for this problem when the right-hand side is in $L^2(\Omega)$. We then state in Subsection~\ref{52} the homogenization and corrector results for the singular semilinear problem \eqref{eqprimah}; the proofs of these results will be given in Section~$9$. 

%fm% nouveau et important
As above we consider in this Section a given matrix $A$ which satisfies \eqref{eq0.0}
and a given function $F$ which satisfies \eqref{car} and \eqref{eq0.1}.
But in this Section, as well as in Section~$9$, we assume that
\begin{equation} \label{N>1}
N \geq 2,
\end{equation}
see Remark~\ref{remN>1} below.

 \subsection {\bf The perforated domains}
\label{51}

 Let $\Omega$ be an open and bounded set of $\mathbb{R}^N$ ($N\geq 2$) and let us perforate it by holes: we obtain an open set $\Omega^\varepsilon$. More precisely, consider for every $\eps$, where $\eps$  takes its values in a sequence of positive numbers which tends to zero, some closed sets $T_i^\eps$ of $\mathbb{R}^N$,  $1\leq i \leq n(\eps)$, which are the holes. The domain $\Omega^\eps$ is defined by removing the holes $T_i^\eps$ from $\Omega$, that is 
\begin{equation*}
\Omega^\eps=\Omega- \bigcup_{i=1}^{n(\eps)}T_i^\eps.
\end{equation*}

 We suppose that the sequence of domains $\oeps$ is such that there exist a sequence of functions $w^\eps$, a distribution $\mu\in\mathcal{D}'(\Omega)$ and two sequences of distributions $\mu^\eps\in\mathcal{D}'(\Omega)$ and $\lambda^\eps\in\mathcal{D}'(\Omega)$ such that
\begin{equation} \label{cond1}
w^\eps\in H^1(\Omega)\cap L^\infty(\Omega),
\end{equation}
\begin{equation} \label{cond1bis}
0\leq w^\eps\leq 1\mbox{ a.e. } x\in\Omega,
\end{equation}
\begin{equation}
\label{cond2}
\forall \psi\in H_0^1(\Omega)\cap L^\infty(\Omega),\,  w^\eps\psi\in H_0^1(\Omega^\eps),
\end{equation}
\begin{equation}
\label{cond3}
w^\eps\rightharpoonup 1\text{  in } H^1(\Omega) \mbox{ weakly, in } L^\infty(\Omega) \mbox{ weakly-star and a.e. in } \Omega, 
\end{equation}
\begin{equation}
\label{cond4}
\mu\in H^{-1}(\Omega),
\end{equation}
\vfill\eject
\begin{align}
\label{cond5}
\begin{cases}
\dys-div \, {}^t\!A(x)Dw^\eps=\mu^\eps-\lambda^\eps\mbox{ in } \mathcal{D}'(\Omega),\\
\mbox{with } \mu^\eps\in H^{-1}(\Omega),\, \lambda^\eps\in H^{-1}(\Omega),\\
\mu^\eps\geq 0 \mbox{ in } \mathcal{D}'(\Omega),\\
\dys \mu^\eps\rightarrow \mu\mbox{ in } H^{-1}(\Omega)\mbox{ strongly},\\
\dys\langle \lambda^\eps,\tilde{z^\eps}\rangle_{H^{-1}(\Omega),H_0^1(\Omega)}=0\,\,\, \forall z^\eps\in H_0^1(\Omega^\eps),
\end{cases}
\end{align}
where, as well as everywhere in the present paper, for every function $z^\eps$ in $L^2(\Omega)$, we define $\tilde{z^\eps}$ as the extension by $0$ of $z^\eps$ to $\Omega$, namely by
\begin{equation}
\label{defex}
\dys\tilde{z^\eps}(x)=\begin{cases}
z^\eps(x) & \text{in}\ \Omega^\eps,\\ 
\dys 0 &\text{in } \dys \bigcup_{i=1}^{n(\eps)}T_i^\eps \, ;
 \end{cases}
\end{equation}
then $\tilde{z^\eps}\in L^2(\Omega)$ and $\|\tilde{z^\eps}\|_{L^2(\Omega)}=\|z^\eps\|_{L^2(\oeps)}$.
Moreover  
\begin{equation}
\begin{cases}
\label{57bis}
\dys\mbox{if } z^\eps\in H_0^1(\oeps),\mbox{ then } \tilde{z^\eps}\in H_0^1(\Omega)\\
\dys\mbox{with } \widetilde{D z^\eps}=D\tilde{z^\eps}\mbox{ and } \|\tilde{z^\eps}\|_{H_0^1(\Omega)}=\|z^\eps\|_{H_0^1(\oeps)}.
\end{cases}
\end{equation}

The meaning of assumption \eqref{cond2} is that
\begin{equation}
\label{57ter}
w^\eps=0 \text{ on } \bigcup_{i=1}^{n(\eps)}T_i^\eps,
\end{equation}
while the meaning of the last statement of \eqref{cond5} is that the distribution $\lambda^\eps$ only acts on the holes $T_i^\eps$, $i=1,\cdots,n(\eps)$, since taking test functions in 
$\mathcal{D}(\oeps)$ in the first statement of \eqref{cond5} implies that
\begin{equation*}
\displaystyle -div\, {}^t\!A(x) D w^\eps  = \mu^\eps\mbox{ in } \mathcal{D}'(\oeps).
\end{equation*}

\bigskip
Taking $z^\eps=w^\eps \phi$, with $\phi\in\mathcal{D}(\Omega)$, $\phi\geq 0$, as test function in \eqref{cond5} we have
$$
\into \phi A(x)Dw^\eps Dw^\eps+\into w^\eps A(x)D w^\eps D\phi=\langle\mu^\eps,w^\eps\phi\rangle_{H^{-1}(\Omega),H_0^1(\Omega)},
$$
from which using \eqref{cond3} and  the fourth statement of \eqref{cond5} we easily deduce that
$$
\into \phi A(x)Dw^\eps Dw^\eps\rightarrow \langle\mu,\phi\rangle_{H^{-1}(\Omega),H_0^1(\Omega)}\,\, \forall\phi\in\mathcal{D}(\Omega),\,\phi\geq0,
$$
and therefore that
\begin{equation*}
\mu\geq 0.
\end{equation*}
The distribution $\mu\in H^{-1}(\Omega)$ is therefore also a nonnegative Radon measure. Moreover, since 
\begin{align*}
\begin{cases}
\forall\phi\in \mathcal{D}(\Omega),\phi\geq 0,\\
%fm%
%\dys\into \phi d\mu=\langle\mu,\phi\rangle_{H^{-1}(\Omega),H_0^1(\Omega)}\leq \limsup_\eps \into \phi A(x)Dw^\eps D w^\eps\leq\\
%\dys \leq \|\phi\|_{L^\infty(\Omega)} \limsup_\eps\into A(x)Dw^\eps D w^\eps\leq C\|\phi\|_{L^\infty(\Omega)},
\dys\into \phi d\mu=\langle\mu,\phi\rangle_{H^{-1}(\Omega),H_0^1(\Omega)}= \lim_\eps \into \phi A(x)Dw^\eps D w^\eps\leq C\|\phi\|_{L^\infty(\Omega)}, 
\end{cases}
\end{align*}
the measure $\mu$  satisfies $\dys\into d\mu \leq C < + \infty$.
%fm% in other terms $\mu \in \mathcal{M}_b(\Omega)$.
\medskip

It is then (well) known$\;$\footnote{\quad the reader who would not enter in this theory could continue reading the present paper assuming in \eqref{cond4} that $\mu$ is a function of $L^r(\Omega)$ (with $r=(2^*)'$ if $N\geq 3$, $r>1$ if $N=2$, and $r=1$ if $N=1$) and not only an element of $H^{-1}(\Omega)$ \label{footnote1}} 
(see e.g. \cite{DMu} Section~1 and \cite{DMu2} Subsection~2.2 for more details) that if $z\in H_0^1(\Omega)$, then $z$ (or more exactly its quasi-continuous representative for the $H_0^1(\Omega)$ capacity) satisfies
\begin{equation}
\label{57bis2}
z\in L^1(\Omega;d\mu)\,\mbox{ with }\,\langle\mu,z\rangle_{H^{-1}(\Omega),H_0^1(\Omega)}=\into z\,d\mu\,;
\end{equation}
moreover if $z\in H_0^1(\Omega)\cap L^\infty(\Omega)$, then $z$ satisfies
\begin{equation}
\label{58bis}
z\in L^\infty(\Omega;d\mu)\, \mbox{ with }\, \|z\|_{L^\infty(\Omega;d\mu)}=\|z\|_{L^\infty(\Omega)};
\end{equation}
therefore when $z\in H_0^1(\Omega)\cap L^\infty(\Omega)$, then $z$ belongs to $L^1(\Omega;d\mu)\cap L^\infty(\Omega;d\mu)$ and therefore to $L^p(\Omega;d\mu)$ for every $p$, $1\leq p\leq +\infty$.

When one assumes that the holes $T_i^\eps$, $i=1,\cdots,n(\eps)$, are such that the assumptions \eqref{cond1}, \eqref{cond1bis}, \eqref{cond2}, \eqref{cond3}, \eqref{cond4} and \eqref{cond5} hold true, then (see \cite{CM} or \cite{MH}, or \cite{DG} for a more general framework) for every $f\in L^2(\Omega)$, the (unique) solution $y^\eps$ to the linear problem
\begin{equation}
\label{521}
\begin{cases}
y^\eps\in H_0^1(\oeps),\\
-div\, A(x)Dy^\eps=f\mbox{ in } \mathcal{D}'(\oeps),
\end{cases}
\end{equation}
satisfies 
\begin{equation*}
\tilde{y^\eps}\rightharpoonup y^0\mbox{ in } H_0^1(\Omega),
\end{equation*}
where $y^0$ is the (unique) solution to 
\begin{equation*}
\begin{cases}
\dys y^0\in H_0^1(\Omega)\cap L^2(\Omega;d\mu),\\
\dys -div\, A(x)Dy^0+\mu y^0=f \mbox{ in } \mathcal{D}'(\Omega),
\end{cases}
\end{equation*}
or equivalently to
\begin{equation}
\label{523}
\begin{cases}
\dys y^0\in H_0^1(\Omega)\cap L^2(\Omega;d\mu),\\
\dys\into A(x)Dy^0Dz+\into y^0 z\,d\mu=\into fz\,\,\, \forall z\in H_0^1(\Omega)\cap L^2(\Omega;d\mu) ;
\end{cases}
\end{equation}
in \eqref{523} a ``strange term" $\mu y^0$,appears; this term is in some sense the asymptotic memory of the fact that $\tilde{y^\eps}$ was zero on the holes.
\bigskip 

The prototype of the examples where assumptions \eqref{cond1}, \eqref{cond1bis},  \eqref{cond2}, \eqref{cond3}, \eqref{cond4} and \eqref{cond5} are satisfied 
is the case where the matrix $A(x)$ is the identity (and therefore where the operator is $-div\, A(x)D=-\Delta$), where $\Omega\subset\mathbb{R}^N$, $N\geq 2$, where the holes $T_i^\eps$ are balls of radius $r^\eps$ (or more generally sets obtained by an homothety of ratio $r^\eps$ from a given bounded closed set\break $T\subset \mathbb{R}^N$) with $r^\eps$ given by
\begin{equation*}
\begin{cases}
r^\eps=C_0\,  \eps^{N/(N-2)}\,\mbox{ if } N\geq 3,\\
%fm% \eps^2\log r^\eps\rightarrow -C_0\, \mbox{ if } N= 2,
r\dys^\eps=\exp ( -\frac{C_0}{\eps^2}) \, \mbox{ if } N= 2,
\end{cases}
\end{equation*}
%fm% (taking $r\dys^\eps=\exp -\frac{C_0}{\eps^2}$ is the model case for $N=2$) 
which are periodically distributed on a $N$-dimensional  lattice of cubes of size $2\eps$, and where the measure $\mu$ is given by 
\begin{equation*}
\begin{cases}
\mu=\dys\frac{S_{N-1}(N-2)}{2^N}C_0^{N-2}\,\mbox{ if } N\geq 3,\\
\mu=\dys\frac{2\pi}{4}\frac{1}{C_0}\, \mbox{ if } N= 2,
\end{cases}
\end{equation*}
(see e.g. \cite{CM} and \cite{MH} for more details, and for other examples, in particular for the case where the holes are distributed on a manifold).

%fm% nouveau
\begin{remark}\label{remN>1}
In this Remark we  prove that in dimension $N = 1$, 
there is no sequence $w^\eps$ which satisfies  \eqref{cond1}, \eqref{cond2} and \eqref{cond3} 
whenever for every $\eps$ there exists at least one hole $T^\eps_i$ with $T^\eps_i \cap \overline \Omega \neq \emptyset$. 
This is the reason why we assume in this Section, as well as in Section~$9$, that $N \geq 2$ (see \eqref{N>1}).

Indeed let $x$ be any point in $\overline \Omega$, and let $z^\eps$ be any point in $T^\eps_i \cap \overline \Omega$.
Assume that there exists a sequence $w^\eps$ which satisfies  \eqref{cond1}, \eqref{cond2} and \eqref{cond3}, 
and let $M > 0$ be such that $\| w^\eps \|_{H^1 (\Omega)}Ê\leq M$ for every $\eps$. 
Since $N =1$, one has $H^1 (\Omega) \subset C^{0, 1/2} (\overline \Omega)$, and since $w^\eps (z^\eps) = 0$, one has
\begin{equation}
\label{holder}
\begin{cases}
\dys | w^\eps (x) | =  | w^\eps (x) - w^\eps (z^\eps) | \leq \| w^\eps \|_{C^{0, 1/2} (\overline \Omega)} \,Ê| x - z^\eps |^{1/2} \leq \\
\dys \leq C \, \| w^\eps \|_{H^1 (\Omega)} \,Ê| x - z^\eps |^{1/2} \leq C MÊ\, | x - z^\eps |^{1/2}.
\end{cases}
\end{equation}
Since there exists a subsequence, still denoted by $\eps$, such that the points $z^\eps$ converge to some point $z \in \overline \Omega$,
passing to the limit in \eqref{holder} gives in view  of \eqref{cond3}
\begin{equation*}
1  \leq C  MÊ\, | x - z |^{1/2} ,
\end{equation*}
which is a contradiction when $x = z$.
\qed
\end{remark}

\subsection {\bf The homogenization result for the singular semilinear  problem (\ref{eqprimah})}
\label{52}
 
 The existence Theorem~\ref{EUS} asserts that when the matrix $A$ and the function $F$ satisfy assumptions \eqref{eq0.0}, \eqref{car} and \eqref{eq0.1}, then for given $\eps>0$, the singular semilinear problem \eqref{eqprimah} posed on $\oeps$ has at least one solution $u^\eps$ in the sense of Definition~\ref{sol} (this solution is moreover unique if the function $F(x,s)$ also satisfies assumption \eqref{eq0.2}).
 
 The following Theorem asserts that the result of the homogenization process for the singular problem \eqref{eqprimah} is very similar to the homogenization process for the linear problem \eqref{521}.

\begin{theorem}
\label{homogenization}
Assume that $N \geq 2$ and that the matrix $A$ and the function $F$ satisfy \eqref{eq0.0}, \eqref{car} and \eqref{eq0.1}. 
Assume also that the sequence of perforated domains $\Omega^\eps$ is such that  \eqref{cond1}, \eqref{cond1bis}, \eqref{cond2}, \eqref{cond3}, \eqref{cond4} and \eqref{cond5} hold true.
Finally let $u^\eps$ be a solution to problem \eqref{eqprimah} in the sense of Definition~\ref{sol}, namely
\begin{equation}
\label{sol1h}
\begin{cases}
i)\, u^\eps\in H^1_{0}(\oeps),\\
ii)\, u^\eps(x)\geq 0 \,\,\mbox{a.e. } x\in\oeps,\\
\end{cases}
\end{equation}
%for every $\varphi^\eps \in H_{0}^{1}(\oeps)$ with $\varphi^\eps \geq 0$ we have
\begin{align}
\label{Leps}
\begin{cases}
\forall  \varphi^\eps \in H_{0}^{1}(\Omega^\eps) \mbox{ with } \varphi^\eps \geq 0, \mbox{ one has }  \\ \vspace{0.1 cm}
\dys\int_{\oeps} F(x,u^\eps) \varphi^\eps<+\infty,\\
\dys\int_{\oeps} A(x)Du^\eps D\varphi^\eps= \int_{\oeps} F(x,u^\eps) \varphi^\eps.
\end{cases}
\end{align}

Then there exists a subsequence, still labelled by $\eps$, such that for this subsequence 
%fm$the solution $u^\eps$ to problem \eqref{eqprimah} in the sense of Definition~\ref{sol}, namely the $u^\eps$ such that 
one has, for $\tilde{u^\eps}$ defined by \eqref{defex},
\begin{equation}
\label{512bis}
 \tilde{u}^\eps\rightharpoonup  u^0 \, \mbox{ in } H^1_0(\Omega)\, \mbox{weakly, }\\
\end{equation}
where $u^0$ is a solution to
\begin{equation}\label{sol1h0}
\begin{cases}
i)\, u^0\in H^1_{0}(\Omega)\cap L^2(\Omega;d\mu),\\
ii)\, u^0(x)\geq 0 \,\,\mbox{a.e. } x\in\Omega,\\
\end{cases}
\end{equation}
%for every  $z\in H_{0}^{1}(\Omega)\cap L^2(\Omega;d\mu)$  with $z \geq 0$ we have
\begin{equation}\label{L_0}
\begin{cases}
\forall  z\in H_{0}^{1}(\Omega)\cap L^2(\Omega;d\mu)   \mbox{ with } z \geq 0, \mbox{ one has }  \\ \vspace{0.1 cm}
\dys\int_\Omega F(x,u^0) z<+\infty, \\
\dys\int_{\Omega} A(x)Du^0 Dz+ \int_\Omega u^0 z d\mu= \int_{\Omega} F(x,u^0) \psi.
\end{cases}
\end{equation}

\end{theorem} 
%\begin{remark}\label{uffa}
%Note that in view of \eqref{57bis} and of the fact that $\mu$ is a nonnegative finite measure which also belongs to $H^{-1}(\Omega)$, for any $k>0$ the function $T_k(u^0)$ (or more exactly its quasi-continuous representative for the $H_0^1(\Omega)$ capacity) is $\mu$-measurable and belongs to $L^\infty(\Omega;d\mu)\cap L^1(\Omega;d\mu)$, and that since $\psi \in H_0^1(\Omega)\cap L^\infty(\Omega)$, also $\psi$ belongs to $L^\infty(\Omega;d\mu)\cap L^1(\Omega;d\mu)$. This implies that $\into T_k(u^0) \psi d\mu$ is correctly defined (and it is equal to  $\langle\mu,T_k(u^0)\psi \rangle_{H^{-1}(\Omega),H_0^1(\Omega)}$ by \eqref{57bis}). Applying Fatou's lemma, the term $\into  u^0 \psi d\mu$ in \eqref{L_0} is also well defined. 
%\qed
%\end{remark}

\begin{remark}
Requirements \eqref{sol1h0} and \eqref{L_0} are the adaptation of the Definition~\ref{sol} of a solution to problem \eqref{eqprima} to the case of problem 
\begin{equation}
\label{518bis}
\begin{cases}
\displaystyle - div \,A(x) D u^0+\mu u^0  = F(x,u^0) & \mbox{in} \; \Omega,\\
u^0 = 0 & \mbox{on} \; \partial \Omega,\\
\end{cases} 
\end{equation}
 in which there is now a zeroth order term $\mu u^0$, where $\mu$ is a nonnegative measure of $H^{-1}(\Omega)$. Theorem~\ref{homogenization} therefore expresses the fact that, when assumptions \eqref{cond1}, \eqref{cond1bis}, \eqref{cond2}, \eqref{cond3}, \eqref{cond4} and \eqref{cond5} hold true, the result of the homogenization process of the singular semilinear problem \eqref{eqprimah} in $\oeps$ with Dirichlet boundary condition on the whole of the boundary $\partial\oeps$ is the singular semilinear problem \eqref{518bis}, where the ``strange term" $\mu u^0$ appears exactly as in the case of the linear problem \eqref{521}  where the right-hand side belongs to $L^2(\Omega)$.
 
  Note nevertheless that the result was not a priori obvious due to the presence of the term $F(x,u^\eps)$, which is singular (at least) on the boundary $\partial\oeps$ and, in particular, on the boundary of the holes, whose number increases more and more when $\eps$ goes to zero, ``invading" the entire open set $\Omega$.
\qed
\end{remark}
\begin{remark}
If $F(x,s)$ satisfies, in addition to \eqref{car} and \eqref{eq0.1}, the further assumption \eqref{eq0.2}, the solution $u^\eps$ to \eqref{sol1h} and \eqref{Leps} is unique (see Theorem~\ref{uniqueness} above), and the solution $u^0$ to \eqref{sol1h0} and \eqref{L_0} is also unique, as it is easily seen from a proof very similar to the one made in Section~\ref{comparison} below.

Under this further assumption there is therefore no need to extract a subsequence in Theorem~\ref{homogenization}, and the convergence takes place for the whole sequence $\eps$.
\qed
\end{remark}

Further to the homogenization result of Theorem~\ref{homogenization}, we will also prove the following corrector result, which, under the assumptions that $u^0\in L^\infty(\Omega)$ and that the matrix $A$ is symmetric, states that $w^\eps u^0$ is a strong approximation in $H_0^1(\Omega)$ of $\tilde{u^\eps}$.

\begin{theorem}
\label{corrector}
Assume that $N \geq 2$ and that the matrix $A$ and the function $F$ satisfy \eqref{eq0.0}, \eqref{car} and \eqref{eq0.1}. Assume also that the sequence of perforated domains $\Omega^\eps$ is such that  \eqref{cond1}, \eqref{cond1bis}, \eqref{cond2}, \eqref{cond3}, \eqref{cond4} and \eqref{cond5} hold true.
Finally let  $u^\eps$ be the subsequence of solutions to problem \eqref{eqprimah} in the sense of Definition~\ref{sol} (see \eqref{sol1h} and \eqref{Leps})
defined in Theorem~\ref{homogenization},
and let $u^0$ be its limit defined by \eqref{512bis}, \eqref{sol1h0} and \eqref{L_0}.
Assume moreover that
\begin{equation}
\label{517a}
A(x)={}^t\!A(x),
\end{equation}
\begin{equation}
\label{517}
u^0\in L^\infty(\Omega).
\end{equation}

Then further to \eqref{512bis} one has
\begin{equation}
\label{518}
\tilde{u^\eps}=w^\eps u^0 +r^\eps,
\mbox{ where } r^\eps\rightarrow 0 \mbox{ in } H_0^1(\Omega) \mbox{ strongly}.
\end{equation}
\end{theorem}

\begin{remark}
\label{rem55}
If further to assumptions \eqref{car} and \eqref{eq0.1}, the function $F$ is assumed to satisfy the regularity assumption \eqref{nuch}, then in view of Proposition~\ref{lem1} every solution $u^\eps$ to \eqref{eqprimah} in the sense of Definition~\ref{sol} satisfies
\begin{equation*}
\begin{cases}\vspace{0.1cm}
\dys\|\tilde{u^\eps}\|_{L^\infty(\Omega)}=\|u^\eps\|_{L^\infty(\oeps)}\leq\\
\dys\leq 1+\frac{2}{\alpha} \,  C(|\oeps|, N, t) \, \|h\|_{L^t(\oeps)}\leq
 1+\frac{2}{\alpha} \,  C(|\Omega|, N, t) \, \|h\|_{L^t(\Omega)}.
\end{cases}
\end{equation*}
In such a case, the limit $u^0$ satisfies assumption \eqref{517} with
$$
\|u^0\|_{L^\infty(\Omega)}\leq 1+\frac{2}{\alpha} \,  C(|\Omega|, N, t) \, \|h\|_{L^t(\Omega)}.
$$

\qed
\end{remark}

 \section{A  priori estimates}

\begin{proposition}{\bf ($H_0^1(\Omega)$ a priori estimate)}.
\label{lem2}
Assume that the matrix $A$ and the function $F$ satisfy \eqref{eq0.0}, \eqref{car} and \eqref{eq0.1}. Then every $u$ solution to problem \eqref{eqprima} in the sense of Definition~\ref{sol} satisfies
%fm% modification de l'estimation !
%\begin{equation}
%\label{num11bis}
%\|u\|_{H_0^1(\Omega)}\leq C(|\Omega|,N,\alpha,\gamma)\left(\|h\|^{\frac{1}{1+\gamma}}_{L^{(2^*)'}(\Omega)}+\|h\|_{L^{(2^*)'}(\Omega)}\right),
%\end{equation}
%where $C(|\Omega|,N,\alpha,\gamma)$ is a increasing function of $|\Omega|$.
\begin{equation}
\label{num11bis}
\|u\|_{H_0^1(\Omega)} \leq C(|\Omega|,N,\alpha,\gamma, r)\,  \big( \|h\|_{L^{r}(\Omega)} + \|h\|_{L^{1}(\Omega)}^{1/2} \big),
\end{equation}
where the constant $C(|\Omega|,N,\alpha,\gamma, r)$ 
depends only on $|\Omega|$, $N$, $\alpha$, $\gamma$ and  $r$ 
and is a nondecreasing function of $|\Omega|$.
\end{proposition}
\begin{proof}
We take $\varphi=u$ as test function in \eqref{L}. 
%fm% nouveau et important
Using (\ref{eq0.1} {\it iv}) 
and Young's inequality with $1/p=1-\gamma$ and $1/p' = \gamma$ when $0<\gamma<1$, which implies that 
\begin{equation}
\label{Younggamma}
u^{1-\gamma} \leq \frac{1}{p} u^{(1-\gamma)p}+\frac{1}{p'} = (1-\gamma) u+\gamma,
\end{equation}
we obtain 
\begin{equation}
\begin{cases}
\label{61bis}
\vspace{0.1 cm}
\dys \int_\Omega A(x) DuDu=\int_\Omega F(x,u) u\leq \into h(x) \big(\frac{1}{u^\gamma}+1\big) u \leq \\ 
\dys \leq \into h(x) \, \big((1 - \gamma ) u + \gamma + u \big) = \into h(x) \, \big((2 - \gamma ) u + \gamma \big).
\end{cases}
\end{equation}

When $N\geq 3$, we use Sobolev's embedding Theorem $H_0^1(\Omega)\subset L^{2^*}(\Omega)$, with $2^*$ defined by $\frac{1}{2^*}=\frac 12-\frac 1N$, and the Sobolev's inequality 
\begin{equation}
\label{Sob}
\|v\|_{L^{2^*}(\Omega)}\leq C_N \|Dv\|_{L^2(\Omega)} \quad \forall v\in H_0^1(\Omega) ;
\end{equation}
note that $(2^*)'=2N/(N+2) = r$ since $N \geq 3$.

%fm% modification importante
%Using \eqref{eq0.0}  and H\"older's inequality in \eqref{61bis} we get
%\begin{equation}
%\label{6.47}
%\begin{cases}
%\dys \alpha\into |Du|^2\leq \|h\|_{L^{(2^*)'}(\Omega)}\left(|\Omega|^{\frac{\gamma}{2^*}}\|u\|^{1-\gamma}_{L^{2^*}(\Omega)}+\|u\|_{L^{2^*}(\Omega)}\right)\leq\\\leq   \|h\|_{L^{(2^*)'}(\Omega)}\left(|\Omega|^{\frac{\gamma}{2^*}}C_N^{1-\gamma}\|Du\|^{1-\gamma}_{L^{2}(\Omega)^N}+C_N\|Du\|_{L^{2}(\Omega)^N}\right),
%\end{cases}
%\end{equation}
%which implies estimate \eqref{num11bis} using Young's inequality with $1/p=1-\gamma$ and $1/p' = \gamma$ when $0<\gamma<1$, namely
%\begin{align*}
%\begin{cases}
%\dys X^{1-\gamma}\leq \frac{1}{p}(\lambda X)^{(1-\gamma)p}+\frac{1}{p'}\left(\frac{1}{\lambda^{1-\gamma}}\right)^{p'}=(1-\gamma)\lambda X+\frac{\gamma}{\lambda^{\frac{1-\gamma}{\gamma}}},\\
%\dys \forall X>0,\, \forall \lambda>0.
%\end{cases}
%\end{align*}
Using in \eqref{61bis} the coercivity \eqref{eq0.0}, H\"older's inequality, Sobolev's inequality  \eqref{Sob} and finally Young's inequality, we get
\begin{equation}
\label{6.47new}
\begin{cases}
\dys \alpha\into |Du|^2\leq (2 - \gamma) \|h\|_{L^{r}(\Omega)} \|u\|_{L^{2^*}(\Omega)} + \gamma \|h\|_{L^{1}(\Omega)} \leq
\\
\dys \leq (2 - \gamma) C_N \|h\|_{L^{r}(\Omega)} \|Du\|_{L^{2}(\Omega)^N} + \gamma \|h\|_{L^{1}(\Omega)} \leq
\\
\dys \leq  {\frac{\alpha}{2}} \|Du\|_{L^{2}(\Omega)^N}^2 + {\frac{1}{2\alpha}} (2 -\gamma)^2 C_N^2 \|h\|_{L^{r}(\Omega)}^2 + \gamma \|h\|_{L^{1}(\Omega)}, 
\end{cases}
\end{equation}
which yields 
\begin{equation*}
 \|Du\|_{L^{2}(\Omega)^N}^2 \leq \Big( {\frac{(2 - \gamma)\,  C_N}{\alpha}}\Big)^2 \, \|h\|_{L^{r}(\Omega)}^2 +  {\frac{2 \gamma}{\alpha}} \, \|h\|_{L^{1}(\Omega)},
 \end{equation*}
 which finally implies
\begin{equation*}
\|Du\|_{L^{2}(\Omega)^N} \leq {\frac{(2 - \gamma) \, C_N}{\alpha}}  \|h\|_{L^{r}(\Omega)} + \Big( {\frac{2 \gamma}{\alpha}} \Big)^{1/2} \, \|h\|_{L^{1}(\Omega)}^{1/2} \, ,
\end{equation*}
namely estimate \eqref{num11bis} with a constant which depends only on $N$, $\alpha$ and $\gamma$.

The proof is similar when $N=1$ and $N=2$, but Sobolev's inequality \eqref{Sob} has now to be replaced by 
\begin{equation*}
\dys \|v\|_{L^{\infty}(\Omega)}\leq | \Omega |^{1/2}  \,  \|Dv\|_{L^2(\Omega)} \quad \forall v\in H_0^1(\Omega) \;\;  \mbox{ when } N= 1,
\end{equation*}
and by
\begin{equation*}
\dys \forall r' > 1, \quad \|v\|_{L^{r'}(\Omega)}\leq C(| \Omega |, r)  \, \|Dv\|_{L^2(\Omega)} \quad \forall v\in H_0^1(\Omega) \;\; \mbox{ when } N= 2,
\end{equation*}
where the constant $C(| \Omega |, r)$ is a nondecreasing function of $| \Omega | $. 

This completes the proof of estimate \eqref{num11bis}.
%fm% fin de la modification
\end{proof}

In the following Proposition we give an estimate of $F(x,u)\varphi$ near the singular set $\{u=0\}$.
To this aim we introduce for $\delta>0$  the  function $Z_\delta:[0,+\infty[\rightarrow[0,+\infty[$ defined by 
\begin{equation}
\label{num23bis}
Z_\delta(s)=\begin{cases}
1, & \mbox{if } 0\leq s\leq \delta, \\
 -\frac{s}{\delta}+2, & \mbox{if }  \delta\leq s\leq 2\delta ,\\
  0, & \mbox{if } 2\delta\leq s.
\end{cases}
\end{equation}

\begin{proposition}{\bf(Control of $\dys\int_{\{u\leq \delta\}}F(x,u)v$ when $\delta$ is small)}.
\label{prop3}
Assume that the matrix $A$ and the function $F$ satisfy \eqref{eq0.0}, \eqref{car} and \eqref{eq0.1}. Then every $u$ solution to problem \eqref{eqprima} in the sense of Definition~\ref{sol} satisfies

\begin{align}
\label{5701bis}
\begin{cases}
\forall \varphi \in H^1_0(\Omega), \,\, \varphi \geq 0, \,\, \forall \delta>0, \\
\dys 0 \leq \int_{\{u\leq\delta\}} F(x,u)\varphi\leq
 \into A(x) Du D\varphi Z_\delta(u).
\end{cases}
\end{align}

\end{proposition}
\begin{proof}
The proof consists in taking $T_k(\varphi)Z_\delta(u)$, $\varphi \in H^1_0(\Omega), \varphi \geq 0$ as test function in \eqref{L}. This function belongs to $H_0^1(\Omega)$ and we get
\begin{equation*}
\begin{cases}
\dys \into A(x) Du DT_k(\varphi) Z_\delta(u)=\\\dys  ={{1}\over{\delta}}\int_{\delta<u<2\delta}A(x) Du Du T_k(\varphi) + \into F(x,u) T_k(\varphi)Z_\delta(u).  
\end{cases}
\end{equation*}
Since $Z_\delta (u) = 1$ on $\{u\leq\delta\}$,
this implies that
\begin{align}
\label{5701bis2}
\begin{cases}
\forall \varphi \in H^1_0(\Omega),\,\,\varphi \geq 0,\,\, \forall k>0, \,\, \forall \delta>0, \\
\dys 0 \leq \int_{\{u\leq\delta\}} F(x,u)T_k(\varphi)\leq
 \into A(x) Du DT_k(\varphi) Z_\delta(u).
\end{cases}
\end{align}

We now pass to the limit in \eqref{5701bis2} as $k$ tends to infinity, using the strong convergence of $D T_k(u)$ to $Du$ in $L^2(\Omega)^N$ in the right-hand side and Fatou's Lemma for on the left-hand side.
This gives \eqref{5701bis}.
\end{proof}

As a consequence of Proposition~\ref{prop3} we have:
\begin{proposition}
\label{prop69}
Assume that the matrix $A$ and the function $F$ satisfy \eqref{eq0.0}, \eqref{car} and \eqref{eq0.1}. Then every $u$ which is solution to problem \eqref{eqprima} in the sense of Definition~\ref{sol} satisfies
\begin{equation}\label{5800}
%F(x,0)=0\,\, \mbox{for a.e.}\,\,x\in\{u=0\}.
%NON, il ne faut pas faire cette modification suggeree par le referee, car on utilise l'assertion
%\int_{\{u=0\}}F(x,u)\varphi=0\quad \forall \varphi \in H^1_0(\Omega),\,\varphi\geq 0.
%dans la demonstration de la proposition 3.4
\int_{\{u=0\}}F(x,u)\varphi=0\quad \forall \varphi \in H^1_0(\Omega),\,\varphi\geq 0.
\end{equation}
\end{proposition}
\begin{proof}
%il en faut pas non plus ecrire ce qui suit, contrairement a ce que suggere le referee
%It is clearly enough to prove that
%\begin{equation}
%\label{5800bis}
%\int_{\{u=0\}}F(x,u)\varphi=0\quad \forall \varphi \in H^1_0(\Omega),\,\varphi\geq 0.
%\end{equation}
Since $\{u=0\}\subset \{u\leq \delta\}$ for every $\delta>0$, inequality \eqref{5701bis} implies that 
\begin{align*}
\label{64bis}
\begin{cases}
\forall \varphi \in H^1_0(\Omega),\,\,\varphi \geq 0,\,\, \forall \delta>0, \\
\dys 0\leq \int_{\{u=0\}} F(x,u)\varphi\leq
 \into A(x) Du D\varphi Z_\delta(u).
\end{cases}
\end{align*}
When $\delta$ tends to zero,
$$
Z_\delta(u)\rightarrow \chi_{_{\{u=0\}}}\,\,\mbox{a.e. in } \Omega, 
$$
but since $u\in H^1_{0}(\Omega)$, one has 
$$
Du=0\,\,\mbox{a.e. on } \{u=0\},
$$
and therefore, since $A(x) Du D\varphi \in L^1(\Omega)$,
$$
\into A(x) Du D\varphi Z_\delta(u)\rightarrow 0\mbox{ as } \delta\to 0.
$$

This proves \eqref{5800}.
\end{proof}

\section{Proofs of the stability, existence  and regularity results\break (TheoremsÊ \ref{est} and \ref{EUS} and Proposition \ref{lem1})}
\label{proofexistence}

\begin{proof}[{\bf Proof of the stability Theorem~\ref{est}}] $\mbox{}$
\medskip

\noindent {\bf First step}

Since all the functions $F_n(x,s)$ satisfy assumptions \eqref{car} and \eqref{eq0.1} for the same $\gamma$ and $h$, every solution $u_n$ to problem \eqref{eqprima}$_n$ in the sense of Definition~\ref{sol} satisfies the a priori estimates \eqref{num11bis} and \eqref{5701bis} of Propositions~\ref{lem2} and \ref{prop3}.

Therefore there exist a subsequence, still labelled by $n$, and a function $u_\infty$ such that
\begin{equation}
\label{70}
u_n \rightharpoonup u_\infty \mbox{ in } H_0^1(\Omega) \mbox{ weakly and a.e. in } \Omega. 
\end{equation}
Since $u_n\geq 0$, we have also $u_\infty\geq 0$.  

Since $u_n$ satisfies \eqref{L}$_n$, we have
\begin{equation}
\label{72}
\int_\Omega A(x) Du_nD\varphi=\int_\Omega F_n(x,u_n) \varphi \,\,\, \forall \varphi\in H_0^1(\Omega),\,\, \varphi\geq 0.
\end{equation}

Using in the left-hand side the weak convergence \eqref{70}, and in the right-hand side
the almost everywhere convergence \eqref{70} of $u_n$ to $u_\infty$, assumption \eqref{num1} on the functions $F_n$ and Fatou's Lemma, one obtains 
 \begin{equation}
 \label{71bis}
\into F_\infty(x,u_\infty)\varphi\leq  \int_\Omega A(x) Du_{\infty}D\varphi <+\infty \,\,\, \forall \varphi\in H_0^1(\Omega),\,\, \varphi\geq 0,
\end{equation}
which implies the first assertion of \eqref{L}$_\infty$.

It remains to prove the second assertion of \eqref{L}$_\infty$. 

\bigskip
\noindent {\bf Second step}

We fix a function $\varphi \in H^1_0(\Omega), \; \varphi \geq 0$, and we
write \eqref{72} as 
\begin{equation}
\label{73}
%\begin{cases}
%\forall \varphi \in H^1_0(\Omega),\,\,\varphi \geq 0,\\
\dys \int_\Omega A(x) Du_nD\varphi=
\displaystyle \int_{\{u_n\leq\delta\}} F_n(x,u_n) \varphi+\int_{\{u_n>\delta\}} F_n(x,u_n) \varphi.
%\end{cases}
\end{equation}

We pass now to the limit as $n$ tends to infinity for $\delta>0$ fixed in \eqref{73}. In the left-hand side we get (as before)

\begin{equation}
\label{74}
\int_\Omega A(x) Du_nD\varphi \to \int_\Omega A(x) Du_\infty D\varphi.
\end{equation}

For what concerns the first term of the right-hand side of \eqref{73} we use the a priori estimate \eqref{5701bis}. Since $D\varphi Z_\delta(u_n)$ tends to $D\varphi Z_\delta(u_\infty)$ strongly in $L^2(\Omega)^N$ while $A(x)Du_n $ tends to $A(x)Du_\infty \,$ weakly in $L^2(\Omega)^N$, we obtain
\begin{equation}
\label{75}
\dys\forall\delta>0,\,\,\limsup_{n} \int_{\{u_n\leq\delta\}} F_n(x,u_n)\varphi\leq
 \into A(x) Du_\infty D\varphi Z_\delta(u_\infty). 
\end{equation}
 
  Since
 $$
Z_\delta(u_\infty)\rightarrow \chi_{_{\{u_\infty=0\}}}\,\,\mbox{a.e. in } \Omega,\mbox{ as } \delta\to 0,
$$
and since $u_{\infty}\in H^1_{0}(\Omega)$ implies that $D u_\infty=0$ almost everywhere on the set  $\{x\in\Omega:u_\infty(x)=0\}$, the right-hand side of \eqref{75} tends to~$0$ when $\delta$ tends to~$0$.

 We have proved that 
\begin{equation}
\label{76}
%\limsup_{\delta}\limsup_{n} \int_{\{u_n\leq\delta\}} F_n(x,u_n)\varphi=0.
\limsup_{n} \int_{\{u_n\leq\delta\}} F_n(x,u_n)\varphi \rightarrow 0 \mbox{ as } \delta \rightarrow 0.
\end{equation}

\bigskip
\noindent {\bf Third step}

Let us now observe that for every $\delta>0$ 
\begin{equation}
\label{74bis}
\int_{\{u_\infty=0\}} F_n(x,u_n)\chi_{_{\{u_n\leq\delta\}}}\varphi\leq\int_{\{u_n\leq\delta\}} F_n(x,u_n)\varphi.
\end{equation} 
Since $u_n$ converges almost everywhere to $u_\infty$, one has, for every $\delta>0$,
$$
\chi_{_{\{u_n\leq\delta\}}}\rightarrow \chi_{_{\{u_\infty\leq\delta\}}} \mbox{ a.e. on } \{x\in\Omega: u_\infty(x)\not=\delta\},
$$
and therefore
$$
\chi_{_{\{u_n\leq\delta\}}}\rightarrow 1\mbox{ a.e. on } \{x\in\Omega:u_\infty(x)=0\},
$$
while in view of assumption \eqref{num1}, one has
$$
F_n(x,u_n(x))\rightarrow F_\infty(x,u_\infty(x)) \mbox{ a.e.  } x\in\Omega.
$$
Applying Fatou's Lemma to the left-hand side of \eqref{74bis},  we obtain 
\begin{align*}
\dys \forall \delta > 0, \; \; \int_{\{u_\infty=0\}} F_\infty(x,u_\infty)\varphi\leq\limsup_{n}\int_{\{u_n\leq\delta\}} F_n(x,u_n)\varphi,
\end{align*}
which in view of \eqref{76} implies that
\begin{equation}
\label{74ter}
\int_{\{u_\infty=0\}}F_\infty(x,u_\infty)\varphi=0.
\end{equation}

\bigskip
\noindent {\bf Fourth step}

 Let us finally pass to the limit in $n$ for $\delta>0$ fixed in the second term of the right-hand side of \eqref{73}, namely in
\begin{align*}
\int_{\{u_n>\delta\}} F_n(x,u_n)\varphi=\into F_n(x,u_n)\chi_{_{\{u_n>\delta\}}}\varphi.
\end{align*}
Since in view of (\ref{eq0.1} {\it iv})
$$
 0\leq F_n(x,u_n)\chi_{_{\{u_n>\delta\}}}\varphi \leq \ h(x)\left(\frac{1}{\delta^\gamma}+1\right)\varphi\,\, \mbox{ a.e. }x\in\Omega,
$$
 since $h \varphi \in L^{1} (\Omega)$,
 since in view of assumption \eqref{num1} and of the almost everywhere convergence \eqref{70} of $u_n$ to $u_\infty$ one has
$$
F_n(x,u_n)\varphi \rightarrow F_\infty(x,u_\infty)\varphi\,\, \mbox{a.e. on } \Omega,
$$
 and finally since
$$
\chi_{_{\{u_n>\delta\}}}\rightarrow \chi_{_{\{u_\infty>\delta\}}} \mbox{ a.e. on } \{x\in\Omega:u_\infty(x)\not=\delta\},
$$
defining the set $\mathcal{C} \subset [0, + \infty [$ by
$$
\mathcal{C} = \{ \delta > 0,  \,    \mbox{ meas}\{x\in\Omega:u_\infty(x)=\delta\}>0 \,  \}
$$
(note that this set is at most countable),
and choosing  $\delta \not\in \mathcal{C}$, Lebesgue's dominated convergence Theorem implies that
\begin{align}
\label{lebs}
\int_{\{u_n>\delta\}} F_n(x,u_n)\varphi\rightarrow \int_{\{u_\infty>\delta\}} F_\infty(x,u_\infty)\varphi\,\mbox{ as } n\to+\infty,\,\,\forall \delta\not\in \mathcal{C}.
\end{align}

Since the set $ \mathcal{C}$ is at most a countable, choosing $\delta$ outside of the set $ \mathcal{C}$ and using the fact that the set $\{x\in\Omega:u_\infty(x)>\delta\}$ monotonically shrinks to the set $\{x\in\Omega:u_\infty(x)>0\}$ as $\delta$ tends to $0$, the fact that $F_\infty(x,u_\infty)\varphi$ belongs to $L^1(\Omega)$ (see \eqref{71bis}), and finally \eqref{74ter}, we have proved that 
\begin{align}
\label{712}
\begin{cases}
\dys\int_{\{u_\infty>\delta\}}F_\infty(x,u_\infty)\varphi\rightarrow\int_{\{u_\infty>0\}}F_\infty(x,u_\infty)\varphi=\into F_\infty(x,u_\infty)\varphi,\\
\mbox{as  } \delta\to 0, \delta\not\in  \mathcal{C}.
\end{cases}
\end{align}

\bigskip
\noindent {\bf Fifth step}

Passing to the limit in each term of \eqref{73}, first in $n$ for $\delta>0$ fixed with $\delta\not\in  \mathcal{C}$, and then for $\delta\not\in  \mathcal{C}$ which tends to $0$, and collecting the results obtained in \eqref{74}, \eqref{76}, \eqref{lebs} and \eqref{712}, we have proved that
\begin{align*}
\int_\Omega A(x) Du_\infty D\varphi= \into F_\infty(x,u_\infty) \varphi \,\,\,\, \forall \varphi\in H^1_0(\Omega),\, \varphi\geq 0,
\end{align*}
which is nothing but the second assertion in \eqref{L}$_\infty$.

We have proved a weaker version of Theorem~\ref{est}, where the strong $H_0^1(\Omega)$ convergence \eqref{num2} is replaced by the weak $H_0^1(\Omega)$ convergence \eqref{70}.

\bigskip
\noindent {\bf Sixth step}

Let us now prove that  \eqref{num2} (namely the strong $H_0^1(\Omega)$ convergence) holds true. 
Indeed, taking $u_n$ as test function in \eqref{L}$_n$, we have
$$
\into A(x) Du_n Du_n=\into F_n(x,u_n)u_n dx.
$$

Observe that in view of hypothesis \eqref{num1} and of convergence \eqref{70} we have
$$
F_n(x,u_n)u_n\rightarrow F_\infty(x,u_\infty)u_\infty \mbox{ a.e. } x\in\Omega,
$$
and that the functions $F_n(x,u_n)$ are equi-integrable: indeed for every\break measurable set $E\subset \Omega$, we have, using (\ref{eq0.1} {\it iv}), \eqref{Younggamma},   H\"older's and Sobolev's inequalities (see the proof of Theorem~\ref{EUS} above)
%fm%
%\begin{equation}
%\begin{cases}
%\label{836n}\vspace{0.1cm}
%\dys 0\leq \int_E F_n(x,u_n) u_n\leq \int_E h(x)\left(\frac{1}{u_n^\gamma}+1\right)u_n\leq\\
%\dys \leq \|h\|_{L^{(2^*)'}(E)}(|\Omega|^{\frac{\gamma}{2^*}}C_N^{1-\gamma}\|Du_n\|_{L^2(\Omega)^N}^{1-\gamma}+C_N\|Du_n\|_{L^2(\Omega)^N})\leq\\
%\dys\leq C \|h\|_{L^{(2^*)'}(E)}.
%\end{cases}
%\end{equation}
\begin{equation}
\begin{cases}
\label{836n}\vspace{0.1cm}
\dys 0\leq \int_E F_n(x,u_n) u_n\leq \int_E h(x)\left(\frac{1}{u_n^\gamma}+1\right)u_n\leq\\
\dys \leq \int_E h(x) \, \big((1 - \gamma ) u_n + \gamma + u_n \big) = \int_E h(x) \, \big((2 - \gamma ) u_n + \gamma \big) \leq \\
\dys \leq (2 - \gamma) \|h\|_{L^{r}(E)} \, C(|\Omega|, N, r) \, \|Du_n\|_{L^2(\Omega)^N} +\gamma \|h\|_{L^{1}(E)} \leq\\
\dys\leq c ( \|h\|_{L^{r}(E)} +  \|h\|_{L^{1}(E)} ),
\end{cases}
\end{equation}
where $c$ is a constant which does not depend on $E$ nor on $n$.
Therefore by Vitali's Theorem 
$$
F_n(x,u_n)u_n\rightarrow F_\infty(x,u_\infty)u_\infty \mbox{ in } L^1(\Omega) \mbox{ strongly}.
$$
On the other hand, taking $u_\infty$ as test function in \eqref{L}$_\infty$, we have
$$
\into A(x) Du_\infty Du_\infty=\into F_\infty(x,u_\infty)u_\infty dx.
$$

Therefore 
$$
\into A(x) Du_n Du_n\rightarrow \into A(x) Du_\infty Du_\infty.
$$
Together with \eqref{70}, this implies the strong convergence \eqref{num2}.

This completes the proof of the stability Theorem~\ref{est}.
\end{proof}

\begin{proof}[{\bf Proof of  the existence Theorem~\ref{EUS}.}]$\mbox{  }$

\indent Let $u_n$ be a solution to 
 \begin{equation}
 \label{81}
\begin{cases}
u_n\in H_0^1(\Omega),\\
\displaystyle - div\, A(x) D u_n  =T_n(F(x,u_n^+)) & \mbox{in} \, \mathcal{D}'(\Omega),
\end{cases} 
\end{equation}
where $T_n$ is the truncation at height $n$.

Since $T_n(F(x,s^+))$ is a bounded Carath\'eodory function defined on\break $\Omega\times \mathbb{R}$, Schauder's fixed point theorem implies that problem \eqref{81} has at least one solution. Since $F(x,s^+)\geq 0$, this solution is nonnegative by the weak maximum principle, and therefore $u_n^+=u_n$.

It is then clear that $u_n$ is a solution to problem \eqref{eqprima}$_n$ in the sense of Definition~\ref{sol},  where \eqref{eqprima}$_n$ is the problem \eqref{eqprima} with $F(x,u)$ replaced by $F_n(x,u_n)=T_n(F(x,u_n))$.

Moreover it is easy to see, considering the cases where $s_\infty>0$ and where $s_\infty=0$, that the functions $F_n(x,s)$ satisfy assumption \eqref{num1} with $$F_{\infty}(x,s)=F(x,s).$$ 
The stability Theorem~\ref{est} then implies that there exists a subsequence of $u_n$ whose limit $u_\infty$ is a solution to problem \eqref{eqprima} in the sense of Definition~\ref{sol}. 

This proves the existence Theorem~\ref{EUS}.
\end{proof}

\begin{proof}[{\bf Proof of  the regularity Proposition~\ref{lem1}.}]$\mbox{  }$

Using $G_k(u)$, $k>0$ as test function in \eqref{L}, we get
$$
\into A(x) DG_k(u)DG_k(u)=\into F(x,u) G_k(u)\quad \forall k>0.
$$
Setting $k=j+1$ with $j\geq 0$, this implies, using the coercivity \eqref{eq0.0} and the growth condition (\ref{eq0.1} {\it iv}), that
\begin{equation}
\label{num11}
\begin{cases}
\dys\alpha\into |DG_{j+1}(u)|^2&\leq\dys \into h(x)\left({{1}\over{u^\gamma}}+1\right)G_{j+1}(u)\leq \\&\dys\leq \int_{\{u>1\}} h(x)\left({{1}\over{u^\gamma}}+1\right)G_{j+1}(u)\leq \\&\dys\leq2\into h(x)\,G_{j+1}(u), \,\, \forall j\geq0.
\end{cases}
\end{equation}

Since $$G_{j+1}(s)=G_j(G_1(s))\quad \forall s\in \mathbb{R}, \, \forall j\geq0,$$
and since $G_1(u)\in H_0^1(\Omega)$,
setting $$\overline{u}=G_1(u),$$
we deduce from \eqref{num11} that $\overline{u}$ satisfies
\begin{equation*}
\begin{cases}
\overline{u}\in H_0^1(\Omega),\\
\displaystyle \into |DG_j(\overline{u})|^2\leq \frac{2}{\alpha}  \into h(x)G_j(\overline{u})\quad \forall j\geq0.
\end{cases}
\end{equation*}

A result of G. Stampacchia (see the proof of Lemma~5.1 and Lemma 4.1 in \cite{S}) (see also Section 5 in \cite{GMM3}) then implies that when $h\in L^t(\Omega)$ (hypothesis \eqref{nuch}), 
the function $\overline{u}$ belongs to $L^\infty (\Omega)$, 
and that there exists a constant $C(|\Omega|, N, t)$ which is nondecreasing in $|\Omega|$ such that
\begin{equation*}
\|\overline{u}\|_{L^\infty(\Omega)}\leq \frac{2}{\alpha} \,  C(|\Omega|, N, t) \, \|h\|_{L^t(\Omega)}.
\end{equation*}

Combined with 
$$
u=T_1(u)+G_1(u)=T_1(u)+\overline{u},
$$
this result implies that 
\begin{equation*}
\|u\|_{L^\infty(\Omega)}\leq 1+ \frac{2}{\alpha} \, C(|\Omega|, N, t) \, \|h\|_{L^t(\Omega)},
\end{equation*}
which proves Proposition~\ref{lem1}.
\end{proof}

\section{Comparison principle and proof of the uniqueness Theorem~\ref{uniqueness}}
\label{comparison}
 In this Section we prove a comparison result, assuming the ``almost nonincreasing monotonicity" condition \eqref{eq0.2} of $F(x,s)$ with respect to $s$. 
\begin{theorem}[\bf Comparison principle]
\label{prop0} 
Assume that the matrix $A$ satisfies \eqref{eq0.0}. Let $F_1(x,s)$ and $F_2(x,s)$ be two functions satisfying \eqref{car} and \eqref{eq0.1}  
(for the same or for different $\gamma$ and $h$). Assume moreover that 
\begin{equation}
\label{80}
\mbox{either } F_1(x,s) \mbox{ or } F_2(x,s) \mbox{ satisfies } \eqref{eq0.2},
\end{equation}
 and that  
\begin{equation}
                              \label{condfc}
F_1(x,s)\leq F_2(x,s)\,\,\, \mbox{a.e. } x\in\Omega,\quad \forall s\geq 0.
\end{equation}

 Let $u_1$ and $u_2$ be solutions in  the sense of Definition~\ref{sol} to problem \eqref{eqprima}$_1$ and \eqref{eqprima}$_2$, where \eqref{eqprima}$_1$ and \eqref{eqprima}$_2$ stand for \eqref{eqprima} with $F(x,u)$ replaced by $F_1(x,u_1)$ and $F_2(x,u_2)$. Then 
\begin{equation}
                               \label{condfc2}
u_1(x)\leq u_2(x) \mbox{ a.e. } x\in\Omega.
\end{equation}
\end{theorem}

\begin{proof}[{\bf Proof of the uniqueness Theorem~\ref{uniqueness}.}]$\mbox{  }$

Applying this comparison principle to the case where $F_1(x,s)=F_2(x,s)=$ $=F(x,s)$, with $F(x,s)$ satisfying \eqref{eq0.2},  immediately proves the uniqueness\break Theorem~\ref{uniqueness}.
\end{proof}

\begin{proof}[\bf Proof of Theorem~\ref{prop0}.]$\mbox{  }$

Since $(u_1-u_2)^+\in H^1_0(\Omega)$, we can take it as test function in \eqref{L}$_1$  and add to both sides  of \eqref{L}$_1$ the finite term $\displaystyle -\lambda \into  u_1(u_1-u_2)^+$. The same holds for \eqref{L}$_2$. This gives 
\begin{align*}
\begin{cases}
\dys\int_\Omega A(x) Du_i D (u_1-u_2)^+ -\lambda \into  u_i(u_1-u_2)^+=\\
\dys= \into (F_i(x,u_i)-\lambda u_i)(u_1-u_2)^+, \,\,\, i=1,2.
\end{cases}
\end{align*}
Taking the difference between these two equations it follows that
\begin{align*}
\begin{cases}
\dys\int_\Omega A(x) D(u_1-u_2)^+D (u_1-u_2)^+-\lambda \into  |(u_1-u_2)^+|^2=\\ 
\dys= \into \big( (F_1(x,u_1)-\lambda u_1)- (F_2(x,u_2)-\lambda u_2) \big) (u_1-u_2)^+ .
\end{cases}
\end{align*}
Using the coercivity \eqref{eq0.0} and the characterization of the first eigenvalue $\lambda_1$ of the operator $-div\,{}^s\! A(x)D$ in $H_0^1(\Omega)$, we get
\begin{align}
\label{rg}
\begin{cases}
(\lambda_1-\lambda) \dys\int_\Omega |(u_1-u_2)^+|^2\leq\\ 
\dys\leq \into \big( (F_1(x,u_1)-\lambda u_1)-(F_2(x,u_2)-\lambda u_2) \big) (u_1-u_2)^+ .
\end{cases}
\end{align}

\bigskip

Let us prove that 
\begin{equation}
\label{734}
\big( (F_1(x,u_1)-\lambda u_1)- (F_2(x,u_2)-\lambda u_2 )\big) (u_1-u_2)^+\leq0 \,\,\mbox{a.e. } x\in\Omega, 
\end{equation}
or equivalently that
\begin{equation}
\label{85bis}
(F_1(x,u_1)-\lambda u_1)- (F_2(x,u_2)-\lambda u_2)\leq0 \,\,\mbox{a.e. } x\in\{u_1>u_2\}.
\end{equation}

We first observe that since $u_1$ and $u_2$ are solutions to \eqref{eqprima}$_1$ and \eqref{eqprima}$_2$ in the sense of Definition~\ref{sol}, one has 
(see \eqref{32bis2})
\begin{equation}
\label{736}
F_1(x,u_1) \mbox{ and } F_2(x,u_2) \mbox{ are nonnegative and finite a.e. } x\in\Omega.
\end{equation}

\bigskip

In order to prove \eqref{85bis}, let us first consider the case where $F_1$ satisfies \eqref{eq0.2}. In this case we have
\begin{equation}
\label{835}
F_1(x,u_1)-\lambda u_1\leq F_1(x,u_2)-\lambda u_2 \,\,\mbox{a.e. } x\in\{u_1>u_2\}.
\end{equation}
We observe that hypothesis \eqref{condfc} implies that 
$$
F_1(x,u_2)\leq F_2(x,u_2)\mbox{ a.e. } x\in\Omega,
$$
and therefore, using \eqref{736}, that 
$$
F_1(x,u_2) \mbox{ is nonnegative and finite a.e. } x\in\Omega.
$$

It is therefore licit to write that 
\begin{equation}
\label{836}
\begin{cases}
(F_1(x,u_1)-\lambda u_1)-(F_2(x,u_2)-\lambda u_2)=\\
= (F_1(x,u_1)-\lambda u_1)-(F_1(x,u_2)-\lambda u_2)\,+\\
+ (F_1(x,u_2)-\lambda u_2)-(F_2(x,u_2)-\lambda u_2) \mbox{ a.e. } x\in\Omega.
\end{cases}
\end{equation}
Since the first line of the right-hand side of \eqref{836} is nonpositive on $\{u_1>u_2\}$ by \eqref{835}, and since the second line of this right-hand side, namely $F_1(x,u_2)-F_2(x,u_2)$, is nonpositive by \eqref{condfc}, we have proved \eqref{85bis}.

\bigskip

Let us now consider the case where $F_2$ satisfies \eqref{eq0.2}. In this case we have 
\begin{equation}
\label{837}
F_2(x,u_1)-\lambda u_1\leq F_2(x,u_2)-\lambda u_2 \,\,\mbox{a.e. } x\in\{u_1>u_2\}.
\end{equation}
We observe that, together with the fact that $F_2(x,u_2)$ is finite almost everywhere on $\Omega$ (see \eqref{736}), this result implies that 
$$
F_2(x,u_1) \mbox{ is nonnegative and finite a.e. } x\in\{u_1>u_2\}.
$$

It is therefore licit to write that 
\begin{equation}
\label{838}
\begin{cases}
(F_1(x,u_1)-\lambda u_1)-(F_2(x,u_2)-\lambda u_2)=\\
= (F_1(x,u_1)-\lambda u_1)-(F_2(x,u_1)-\lambda u_1)\,+\\
+ (F_2(x,u_1)-\lambda u_1)-(F_2(x,u_2)-\lambda u_2) \mbox{ a.e. }x\in\{u_1>u_2\}.
\end{cases}
\end{equation}
Since the second line of the right-hand side of \eqref{838} is nonpositive on $\{u_1>u_2\}$ by \eqref{837}, and since the first line of this right-hand side, namely $F_1(x,u_1)-F_2(x,u_1)$, is nonpositive by \eqref{condfc}, we have again proved \eqref{85bis}.

\bigskip

In both cases we have proved that the right-hand side of \eqref{rg} is nonpositive when assumptions \eqref{80} and \eqref{condfc} are assumed to hold true. Since $\lambda_1-\lambda>0$ by hypothesis \eqref{eq0.2}, this implies that  $(u_1-u_2)^+=0$. 

This proves \eqref{condfc2}.
\end{proof}

\begin{remark}
\label{rem82}
Consider the case where the matrix $A$ satisfies \eqref{eq0.0} and is symmetric and where the function $F$ is defined by
\begin{equation}
\label{unf2}
F(x,s)= \lambda_1 T_k(s) \quad \forall s \geq 0,
\end{equation}
where $T_k$ is the truncation at height $k>0$, for some $k$ fixed, and where $\lambda_1$ and $\phi_1$ are the first eigenvalue and eigenvector of the operator  $-div \,A(x) D$ in $H_0^1(\Omega)$, namely
\begin{equation}
\label{unf3}
\begin{cases}
\dys\phi_1\in H_0^1(\Omega),\, \phi_1\geq 0, \, \into |\phi_1|^2=1,\\
\dys -div\, A(x)D\phi_1=\lambda_1\phi_1 \;\;\;\mbox{  in } \mathcal{D'}(\Omega).
\end{cases}
\end{equation}
The function $F$ defined by \eqref{unf2} satisfies assumptions \eqref{car}, \eqref{eq0.1} and \eqref{unf}, but does not satisfy \eqref{eq0.2}.

Recall that $\phi_1$, the solution to \eqref{unf3}, belongs to $L^\infty(\Omega)$. Then for every $t$ with %$0\leq s\leq \frac{k}{\|\phi_1\|_{L^\infty(\Omega)}}$,  
$0\leq t \leq  {k} / {\|\phi_1\|_{L^\infty(\Omega)}}$,
the function 
$$
u=t\phi_1
$$
is a solution to \eqref{eqprima} in the classical weak sense, and therefore in the sense of Definition~\ref{sol}. 

This proves that uniqueness does not hold if assumption \eqref{eq0.2} is replaced by the weaker assumption \eqref{unf}.
\qed

\end{remark}

\section{Proofs of the homogenization Theorem~\ref{homogenization} and of the corrector Theorem~\ref{corrector}}

\begin{proof}[{\bf Proof of the homogenization Theorem~\ref{homogenization}}] $\mbox{}$
\medskip

\noindent {\bf First step}

Theorem~\ref{EUS} asserts that  for every $\eps>0$ there exists at least one solution to problem \eqref{eqprimah} in the sense of Definition~\ref{sol}, namely a least one $u^\eps$ which satisfies \eqref{sol1h} and \eqref{Leps}.

Proposition~\ref{lem2} implies that every such $u^\eps$ satisfies
%fm%
%\begin{equation}
%\label{551}
%\begin{cases}
%\|\tilde u^\eps\|_{H_0^1(\Omega)}=\|u^\eps\|_{H_0^1(\Omega^\eps)}\leq\\
%\leq C(|\Omega^\eps|, N, \alpha, \gamma) \left(\|h\|^{\frac{1}{1+\gamma}}_{L^{(2^*)'}(\Omega^\eps)}+\|h\|_{L^{(2^*)'}(\Omega^\eps)}\right)\leq\\
%\leq C(|\Omega|, N, \alpha, \gamma) \left(\|h\|^{\frac{1}{1+\gamma}}_{L^{(2^*)'}(\Omega)}+\|h\|_{L^{(2^*)'}(\Omega)}\right).
%\end{cases}
%\end{equation}
\begin{equation}
\label{551}
\begin{cases}
\|\tilde u^\eps\|_{H_0^1(\Omega)}=\|u^\eps\|_{H_0^1(\Omega^\eps)} \leq\\
\leq C(|\Omega^\eps |,N,\alpha,\gamma, r) \, \big( \|h\|_{L^{r}(\Omega^\eps)} + \|h\|_{L^{1}(\Omega^\eps)}^{1/2} \big)  \leq\\
\leq C(|\Omega|,N,\alpha,\gamma, r) \, \big( \|h\|_{L^{r}(\Omega)} + \|h\|_{L^{1}(\Omega)}^{1/2} \big).
\end{cases}
\end{equation}

Estimate \eqref{551} implies that there exists a function $u^0$, and  a subsequence $\tilde{u^\eps}$, still labelled by $\eps$, which  satisfies 
\begin{equation}
\label{ae}
\tilde{u^\eps}\rightharpoonup u^0  \text{  in } H_0^1(\Omega)  \text{  weakly  and a.e. in } \Omega.
\end{equation}

Observe that $u^0(x)\geq 0$ a.e. $x\in\Omega$.

\bigskip
\noindent {\bf Second step}

In view of assumptions \eqref{cond1}, \eqref{cond1bis} and \eqref{cond2}, one has
$$
w^\eps \psi \in H_0^1(\oeps)\cap L^\infty(\oeps) \quad \forall\psi \in H_0^1(\Omega)\cap L^\infty(\Omega),
$$
and
\begin{align*}
\begin{cases}
\|w^\eps \psi  \|_{H_0^1(\oeps)}=\|w^\eps \psi  \|_{H_0^1(\Omega)}\leq\\
\leq \|w^\eps\|_{L^\infty(\Omega)}\|D\psi\|_{L^2(\Omega)^N}+\|\psi\|_{L^\infty(\Omega)}\|Dw^\eps\|_{L^2(\Omega)^N}\leq\\
\leq C^*(\|D\psi\|_{L^2(\Omega)^N}+\|\psi\|_{L^\infty(\Omega)}),
\end{cases}
\end{align*}
where
 \begin{equation*}
 C^*=\max_\eps \{1,\|Dw^\eps\|_{L^2(\Omega)^N}\}.
 \end{equation*} 
 
We now fix $\psi\in H^1_0(\Omega)\cap L^\infty(\Omega),\,\, \psi \geq 0$, and we use $\varphi^\eps=w^\eps \psi \in H^1_0(\Omega^\eps)$, $w^\eps \psi\geq 0$, as test function in \eqref{Leps}. We obtain
\begin{equation*}
\int_{\oeps} A(x)D u^\eps D\psi \,w^\eps +\int_{\oeps} A(x)D u^\eps Dw^\eps\, \psi=\int_{\oeps} F(x,u^\eps) w^\eps \psi,
\end{equation*}
which using \eqref{57bis} implies that
\begin{equation}\label {aa}
\int_{\Omega} A(x)D \tilde{u^\eps} D\psi \,w^\eps +\int_{\Omega} A(x)D \tilde{u^\eps} Dw^\eps\, \psi=\int_{\Omega} \widetilde{F(x,u^\eps)} w^\eps \psi.
\end{equation}

Equation \eqref{aa} in particular implies by \eqref{551} and \eqref{cond3} that 
\begin{equation}\label {bound}
\int_{\Omega}\widetilde{F(x,u^\eps) }w^\eps \psi\leq C
\end{equation}
where $C$ is independent of $\eps$. 

We now claim that for a subsequence, still labelled by $\eps$,
\begin{equation}\label {concar}
\chi_{_{\Omega^{\eps}}}\rightarrow 1\,\,\, \mbox {a.e. in} \,\,\,\Omega;
\end{equation}
indeed, from $w^\eps \chi_{_{\Omega^{\eps}}}=w^\eps\,\, \mbox {a.e. in} \,\,\,\Omega$, which results from \eqref{cond2} (see also  \eqref{57ter}), and from \eqref{cond3} we  get
\begin{align*}
\begin{cases}
\dys\chi_{_{\Omega^{\eps}}}=\chi_{_{\Omega^{\eps}}}w^\eps+\chi_{_{\Omega^{\eps}}}(1-w^\eps)=w^\eps+\chi_{_{\Omega^{\eps}}}(1-w^\eps) \rightharpoonup  1 \\
\dys\mbox{in } L^{\infty}(\Omega) \mbox{ weakly-star}, 
\end{cases}
\end{align*}
which implies that
$$
\into |\chi_{_{\Omega^{\eps}}}-1|=\into (1-\chi_{_{\Omega^{\eps}}})  \rightarrow 0,
$$
which implies \eqref{concar} (for a subsequence).

We deduce from \eqref{concar} that for almost every $x_0$ fixed in $\Omega$ there exists $\eps_0(x_0)$ such that $ \chi_{_{\Omega^{\eps}}}(x_0)=1$ for every $\eps \leq \eps_0(x_0)$, which means that $x_0 \in \Omega^{\eps}$ for every $\eps \leq \eps_0(x_0).$ This implies that  
$$
\widetilde{F(x,u^\eps)}(x_0)=F(x,u^\eps)(x_0)=F(x,\tilde{u^\eps})(x_0)\quad  \forall \eps \leq \eps_0(x_0).
$$
Therefore, using   \eqref{ae}, we get 
\begin{equation*}
\widetilde{F(x,u^\eps)}(x_0)=F(x,\tilde{u^\eps}(x_0))\rightarrow F(x,u^0 (x_0)) \mbox{  as } \eps \rightarrow 0,
\end{equation*}   
or in other terms 
\begin{equation}
\label{good}
\widetilde{F(x,u^\eps)}\rightarrow F(x,u^0)  \mbox{ a.e. } x\in\Omega.
\end{equation}

Using \eqref{bound}, \eqref{cond3} and \eqref{good}  and applying Fatou's Lemma implies that
\begin{equation}\label{stimal1}
\int_{\Omega}F(x,u^0)  \psi<+\infty\,\,\,\forall \psi\in H_0^1(\Omega)\cap L^\infty(\Omega),\,\,\psi\geq0.
\end{equation}

\bigskip
\noindent {\bf Third step}

Let us now fix $\phi \in \mathcal{D}(\Omega), \phi \geq 0$,
and take $\psi = \phi$ in \eqref{aa}.
Since in view of \eqref{cond5} one has
\begin{equation*}
\begin{cases}
\dys\int_{\Omega} A(x) D\tilde{u^\eps} D w^\eps\, \phi=\int_{\Omega} {}^t\!A(x)D{w^\eps} D(\phi \tilde{u^\eps})-\int_{\Omega} {}^t\!A(x)D{w^\eps} D\phi\, \tilde{u^\eps}=\\\dys=\langle \mu^\eps,\phi \tilde{u^\eps}\rangle_{H^{-1}(\Omega),H_0^1(\Omega)}-\int_{\Omega} {}^t\!A(x)D{w^\eps} D\phi\, \tilde{u^\eps},
\end{cases}
\end{equation*}
equation \eqref{aa} implies that  
\begin{align}
\label{esth5}
\begin{cases}
\dys\int_{\Omega} A(x)D \tilde{u^\eps} D\phi\, w^\eps+\langle \mu^\eps,\phi \tilde{u^\eps} \rangle_{H^{-1}(\Omega),H_0^1(\Omega)}
\dys-\int_{\Omega} {}^t\!A(x)D{w^\eps} D\phi \,\tilde{u^\eps}=\\\dys=\int_{\Omega}\widetilde{F(x,u^\eps)} w^\eps \phi \quad
\forall \phi\in \mathcal{D}(\Omega), \phi \geq 0.
\end{cases}
\end{align}

 Using \eqref{ae}, \eqref{cond2}, \eqref{cond3} and \eqref{cond5}, we can easily pass to the limit in the left-hand side of \eqref{esth5}, and we obtain 
 %fm% in view of \eqref{57bis2} (see footnote$^{(1)}$)
\begin{align}
\label{zv2}
\begin{cases}
\dys \int_{\Omega} A(x)D \tilde{u^\eps} D\phi\, w^\eps+\langle \mu^\eps, \phi\tilde{u^\eps} \rangle_{H^{-1}(\Omega),H_0^1(\Omega)}
\dys-\int_{\Omega} {}^t\!A(x)D{w^\eps} D\phi\, \tilde{u^\eps}\rightarrow
\\\dys\rightarrow\int_{\Omega} A(x)D u^0 D\phi+\langle \mu,\phi u^0\rangle_{H^{-1}(\Omega),H_0^1(\Omega)}.
%fm% =\\
% \dys = \into A(x) Du^0D\phi+\into u^0\phi d\mu.
\end{cases}
\end{align}

As far as the right-hand side of \eqref{esth5} is concerned we split it for every $\delta>0$ as
\begin{equation}\label{rr}
\int_{\Omega} \widetilde {F(x,u^\eps)}w^\eps \phi=\int_\Omega \widetilde{F(x,u^\eps)}w^\eps \phi \chi_{_{\{0\leq \tilde{u^\eps} \leq \delta\}}}  +\int_\Omega \widetilde{F(x,u^\eps)}w^\eps \phi \chi_{_{\{  \tilde{u^\eps} > \delta\}}}.
\end{equation}

\bigskip
\noindent {\bf Fourth step}

We now use $\phi^\eps= w^\eps \phi Z_{\delta}(u^\eps)$ as test function in \eqref{Leps}, where the function $Z_{\delta}$ is defined by \eqref{num23bis} and where $\phi \in \mathcal{D}(\Omega)$, $\phi\geq 0$. Note that $\phi^\eps\in H_0^1(\oeps) \cap $ $\cap L^\infty (\oeps)$, $\phi^\eps\geq 0$ in view of \eqref{cond2}. We get
\begin{align*}
\begin{cases}
\dys\int_{\oeps} F(x,u^\eps) w^\eps \phi Z_{\delta}(u^\eps)=\\\vspace{0.1cm}
\dys=\int_{\oeps} A(x) Du^\eps D w^\eps\, \phi Z_\delta(u^\eps) + \int_{\oeps}A(x)D u^\eps D\phi\, w^\eps Z_{\delta}(u^\eps) \, + \\
\dys+\int_{\oeps} A(x) Du^\eps D u^\eps\, Z'_\delta(u^\eps)w^\eps\phi,
 \end{cases}
\end{align*}
which implies, since $Z_\delta(s)=1$ for $0\leq s\leq \delta$ and  since $Z_\delta$ is nonincreasing, that
\begin{align*}
\begin{cases}
\dys\int_{\oeps} F(x,u^\eps) w^\eps \phi \chi_{_{\{0\leq u^\eps \leq \delta\}}} \leq\\
 \dys\leq \int_{\oeps} A(x) Du^\eps D w^\eps\, \phi Z_\delta(u^\eps)+\int_{\oeps} A(x)D u^\eps D\phi \,w^\eps Z_{\delta}(u^\eps).
 \end{cases}
\end{align*}
In view of the definition \eqref{defex} of the extension by zero and of \eqref{57bis}, we get
\begin{align}
\label{912}
\begin{cases}
\dys\int_{\Omega} \widetilde{F(x,u^\eps)} w^\eps \phi \chi_{_{\{0\leq \tilde{u^\eps} \leq \delta\}}} \leq\\
 \dys\leq \int_{\Omega} A(x) D\tilde{u^\eps} D w^\eps\, \phi Z_\delta(\tilde{u^\eps})+\int_{\Omega} A(x)D \tilde{u^\eps} D\phi \,w^\eps Z_{\delta}(\tilde{u^\eps}).
 \end{cases}
\end{align}

Let us define the function $Y_\delta : [0,+\infty[\rightarrow[0,+\infty[$ by
$$
Y_\delta(s)=\int_0^s Z_\delta(\sigma) d\sigma,\quad \forall s\geq 0,
$$
and observe that  $Y_\delta(u^\eps)\in H_0^1(\oeps)$ and $\widetilde{Y_\delta(u^\eps)}=Y_\delta(\tilde{u^\eps})$. Using \eqref{cond5}, we have
\begin{align} 
\label{zv1}
\begin{cases}
\dys  \into A(x) D\tilde{u^\eps} D w^\eps \phi Z_\delta(\tilde{u^\eps}) = \into{}^t\! A(x) D w^\eps  DY_\delta(\tilde{u^\eps})\phi=\\
\dys= \into{}^t\! A(x) D w^\eps  D(\phi Y_\delta(\tilde{u^\eps}))-\into{}^t\! A(x) D w^\eps  D\phi\, Y_\delta(\tilde{u^\eps})=
\\ \dys= \langle \mu^\eps,\phi Y_{\delta}(\tilde{u^\eps}) \rangle_{H^{-1}(\Omega),H_0^1(\Omega)}- \into {}^t\!A(x) Dw^\eps D\phi \,Y_\delta(\tilde{u^\eps}).
\end{cases}
\end{align}

Using now \eqref{cond5}, \eqref{ae}, the fact that
$$
Y_\delta(\tilde{u^\eps})\rightharpoonup Y_\delta(u^0)\mbox{ in } H_0^1(\Omega) \mbox{ weakly and } L^2 (\Omega) \mbox{ strongly},
$$
and \eqref{cond3} proves that the right-hand side of \eqref{zv1} tends to
$$
\langle\mu, \phi Y_\delta(u^0)\rangle_{H^{-1}(\Omega), H_0^1(\Omega)}
$$ 
as $\eps$ tends to zero for $\delta>0$ fixed.

Turning back to \eqref{912}, using \eqref{zv1} and the latest result, and passing to the limit in the last term of \eqref{912}, 
we have proved that for every fixed $\delta>0$ 
\begin{align} 
\label{esth4}
\begin{cases}
\dys\limsup_{\eps}\int_{\Omega} \widetilde{F(x,u^\eps)} w^\eps \phi \chi_{_{\{0\leq \tilde{u^\eps} \leq \delta\}}} \leq
\\\dys\leq \langle \mu,\phi Y_{\delta}(u^0) \rangle_{H^{-1}(\Omega),H_0^1(\Omega)}+\int_{\Omega} A(x)Du^0D\phi\, Z_{\delta}(u).
\end{cases}
\end{align}

We now pass to the limit  in \eqref{esth4} as $\delta$ tends to zero.

For the first term of the right-hand side of \eqref{esth4}, 
%fm% we use \eqref{57bis2} (see footnote$^{(1)}$) and the fact that $0\leq Y_\delta(s)\leq \frac 32 \delta$ for every $s\geq0$; we get
we use the fact that
\begin{equation}
\label{9.200}
0\leq Z_\delta(u^0)\leq 1,\quad Z_\delta(u^0)\rightarrow \chi_{\{u^0=0\}} \,\, \mbox{ a.e. in } \Omega  \quad \mbox{ as } \delta\to 0,
\end{equation}
and
\begin{equation}
\label{9.201}
Du^0=0   \,\,\mbox{Ê a.e. }  \,\, x \in \{ x \in \Omega:u^0(x)=0\} \,\,\mbox{Ê since } \,\,  u^0\in H^1_0(\Omega),
\end{equation}
imply that 
\begin{equation*}
D Y_\delta (u^0) =  Z_\delta (u^0) D u^0 \rightarrow \chi_{ \{u^0=0\} } D u^0 = 0  \;\; \mbox{ strongly in }  L^2 (\Omega)^N ; 
\end{equation*}
this implies the strong $H^1_0 (\Omega)$ convergence of $Y_\delta (u^0)$ to $0$, and therefore that
\begin{align*}
%fm%0\leq \langle \mu,\phi Y_\delta(u^0)\rangle_{H^{-1}(\Omega),H_0^1(\Omega)}= \into \phi Y_\delta(u^0) d\mu\leq \frac 32 \delta \into \phi d\mu\rightarrow 0 \mbox{ as }  \delta\to 0.
\langle \mu,\phi Y_\delta(u^0)\rangle_{H^{-1}(\Omega),H_0^1(\Omega)} \rightarrow 0 \quad \mbox{ as }  \delta\to 0.
\end{align*}

For the second term of the right-hand side of \eqref{esth4} we have, using again \eqref{9.200} and \eqref{9.201},

$$
\int_{\Omega} A(x)Du^0D\phi \,  Z_{\delta}(u^0)\rightarrow \int_{\Omega} A(x)Du^0D\phi  \, \chi_{\{u^0=0\}}=0  \mbox{ as }  \delta\to 0.
$$
%fm%
%which results from to the fact that
% $$
%0\leq Z_\delta(u^0)\leq 1,\quad Z_\delta(u^0)\rightarrow \chi_{\{u^0=0\}}\,\,\mbox{a.e. in } \Omega,\mbox{ as } \delta\to 0,
%$$
%and then from the fact that $Du^0=0$ on the set $\{u^0=0\}$ since  $u^0\in H^1_0(\Omega)$.

As far as the first term of the right-hand side of  \eqref{rr} is concerned, we have we proved that
\begin{equation}
\label{home}
%\lim_{\delta}\limsup_{\eps} \int_\Omega \widetilde {F(x,{u^\eps})} w^\eps \phi\chi_{_{\{0\leq \tilde{u^\eps} \leq \delta\}}} = 0.
\limsup_{\eps} \int_\Omega \widetilde {F(x,{u^\eps})} w^\eps \phi\chi_{_{\{0\leq \tilde{u^\eps} \leq \delta\}}} \rightarrow 0 \quad \mbox{ as }  \delta\to 0.
\end{equation}

\bigskip
\noindent {\bf Fifth step}

Let us now pass to the limit in the second term of the right-hand side of  \eqref{rr}.

Observe that there is at most a countable set $\mathcal{C}^0$ of values of $\delta > 0$ such that 
$$
\mbox{meas}\{x\in\Omega:u^0(x)=\delta\}>0 \mbox{ if } \delta\in \mathcal{C}^0.
$$

From now on we will often choose $\delta > 0$ outside of  this set  $\mathcal{C}^0$.

Using \eqref{good}, \eqref{cond3}, \eqref{ae}, the fact that
$$
\forall\delta>0,\, \chi_{_{\{ \tilde{u^\eps} > \delta\}}}\rightarrow \chi_{_{\{u^0 > \delta\}}}\mbox{ a.e. } x\not\in\{u^0=\delta\},
$$
and therefore that
$$
\forall\delta\not\in \mathcal{C}^0,\,  \chi_{_{\{ \tilde{u^\eps} > \delta\}}}\rightarrow \chi_{_{\{u^0 > \delta\}}}\mbox{ a.e. } x\in\Omega,
$$
and the estimate (see (\ref{eq0.1} {\it iv})), which yields
\begin{align*}
\begin{cases}
\dys0\leq\widetilde{F(x,u^\eps)}w^\eps\phi(x) \chi_{_{\{ \tilde{u^\eps} > \delta\}}}\leq\\
\dys \leq h(x)\left(\frac{1}{(\tilde{u^\eps})^\gamma}+1\right)\phi(x)\chi_{_{\{ \tilde{u^\eps} > \delta\}}}\leq h(x)\left(\frac{1}{\delta^\gamma}+1\right)\phi(x)\mbox{ a.e. }x\in\Omega,
\end{cases}
\end{align*}
Lebesgue's dominated convergence Theorem implies that
\begin{equation*}
\lim_{\eps}\int_\Omega \widetilde{F(x,u^\eps)}w^\eps \phi \chi_{_{\{ \tilde{u^\eps} > \delta\}}} =\int_\Omega F(x,u^0)\phi \chi_{_{\{u^0> \delta\}}} \, \forall\delta\not\in \mathcal{C}^0.
\end{equation*}
Using \eqref{stimal1}  and Lebesgue's dominated convergence Theorem, we  pass to the limit in this equality when $\delta\not\in \mathcal{C}^0$ tends to zero. We obtain 
\begin{equation}
\label{convf2}
\lim_{\delta\not\in \mathcal{C}^0}\lim_{\eps}\int_\Omega \widetilde{F(x,u^\eps)}w^\eps \phi\chi_{_{\{ \tilde{u^\eps} > \delta\}}} =\into F(x,u^0)\phi \chi_{_{\{ u^0> 0\}}} .
\end{equation}

We now want to prove that 
\begin{equation}\label{subfinal} 
\int_{_{\{u^0=0\}}} F(x,u^0) \phi=0.
\end{equation}

Since $\tilde{u^\eps}$ converges almost everywhere to $u^0$, 
one has,  for almost every $x_0\in\{x\in\Omega:u^0(x)=0\}$,
 $$
 \tilde{u^\eps}(x_0)\to 0\mbox{ as } \eps\to 0,
 $$ 
 and therefore $\tilde{u^\eps}(x_0)<\delta$ for every $\eps<\eps_0(x_0)$. This implies that for every $\delta > 0 $
$$
\chi_{\{0\leq \tilde{u^\eps} \leq \delta\}}\rightarrow 1 \mbox{ a.e. } x\in \{x\in\Omega:u^0(x)=0\}.
$$

Using this fact, \eqref{good}, \eqref{cond3} and Fatou's Lemma
 for $\delta>0$ fixed we get
$$
\int_{_{\{u^0=0\}}} F(x,u^0) \phi \leq \liminf_{\eps} \int_{\{u^0=0\}} \widetilde {F(x,{u^\eps})} w^\eps \phi\chi_{_{\{0\leq \tilde{u^\eps} \leq \delta\}}} \,\forall\delta>0,
$$
which, passing to the limit with $\delta$ which tends to zero and using \eqref{home} gives \eqref{subfinal}. This implies that
\begin{equation}\label{final} 
\into F(x,u^0)\phi\chi_{_{\{ u^0> 0\}}} =\int_\Omega F(x,u^0)\phi.
\end{equation}

\bigskip
\noindent {\bf Sixth step}

We come back to \eqref{esth5}. Collecting together \eqref{zv2}, \eqref{rr}, \eqref{home}, \eqref{convf2}  and \eqref{final} we have proved that
\begin{equation}
\label{antefinalreg}
\begin{cases}
\forall \phi \in \mathcal{D}(\Omega),\, \phi \geq 0,\\
\dys\int_{\Omega} A(x)Du^0 D\phi+  \langle \mu, u^0 \phi \rangle_{H^{-1}(\Omega),H_0^1(\Omega)}= \int_{\Omega} F(x,u^0) \phi.
\end{cases}
\end{equation}
Using \eqref{57bis2} (see footnote$\;^{\ref{footnote1}}$), this is equivalent to
\begin{equation}
\label{antefinalregbis}
\begin{cases}
\forall \phi \in \mathcal{D}(\Omega),\, \phi \geq 0,\\
\dys\int_{\Omega} A(x)Du^0 D\phi+ \int_\Omega u^0 \phi d\mu= \int_{\Omega} F(x,u^0) \phi.
\end{cases}
\end{equation}

\bigskip
\noindent {\bf Seventh step}

Let us now take $\psi\in H_0^1(\Omega)\cap L^\infty (\Omega)$, $\psi\geq0$. 

Consider a sequence $\psi_n$ such that
\begin{align*}
\begin{cases}
\dys\psi_n\in\mathcal{D}(\Omega),\, \psi_n\geq 0, \,\|\psi_n\|_{L^\infty(\Omega)}\leq C,\\
\dys \psi_n\rightarrow \psi \, H_0^1(\Omega) \mbox{ strongly, a.e. } x\in \Omega,
%\dys \mbox{and quasi-everywhere in } \Omega \mbox{ for the } H_0^1(\Omega) \mbox{ capacity},
\end{cases}
\end{align*}
and define  
$$
\hat\psi_n=\inf\{\psi_n, \psi\};
$$
then
\begin{align*}
\begin{cases}
\dys\hat\psi_n\in H_0^1(\Omega)\cap L^\infty (\Omega),\, \hat\psi_n\geq 0, \,\|\hat\psi_n\|_{L^\infty(\Omega)}\leq C,\\
\dys \supp\, \hat\psi_n\subset \supp\, \psi_n\subset\subset \Omega,\\
\dys \hat\psi_n \rightarrow \psi \mbox{ in } H_0^1(\Omega) \mbox{ strongly and (for a susequence)  a.e. } x\in \Omega. %fm%,\\
%\dys \mbox{and quasi-everywhere in } \Omega \mbox{ for the } H_0^1(\Omega) \mbox{ capacity}.
\end{cases}
\end{align*}

For the moment let $n$ be fixed and let $\rho_\eta$ be a sequence of mollifiers. For $\eta$ sufficiently small the support of $\hat\psi_n\star \rho_\eta$ is included in a fixed compact $K_n$ of $\Omega$, and $\hat\psi_n\star \rho_\eta\in\mathcal{D}(\Omega)$, $\hat\psi_n\star \rho_\eta\geq 0$. We can therefore use $\phi=\hat\psi_n\star \rho_\eta$ as test function in \eqref{antefinalregbis}. We get
$$
\into A(x)Du^0 D(\hat\psi_n\star \rho_\eta)+\into u^0(\hat\psi_n\star \rho_\eta)d\mu=\into F(x,u^0)(\hat\psi_n\star \rho_\eta).
$$
 Let us pass to the limit in each term of this equation for $n$ fixed as $\eta$ tends to zero. In the right-hand side we use the facts that $F(x,u^0)\in L^1_{\mbox{{\tiny loc}}}(\Omega)$ (see \eqref{stimal1}), that $\supp \, (\hat\psi_n\star \rho_\eta)\subset K_n$, that $\|\hat\psi_n\star \rho_\eta\|_{{L^\infty}(\Omega)}
 \leq \|\hat\psi_n\|_{{L^\infty}(\Omega)}$ and the almost convergence of $\hat\psi_n\star \rho_\eta$ to $\hat\psi_n$,
 and we apply Lebesgue's dominated convergence Theorem. In the first term of the left-hand side we use the strong convergence of $\hat\psi_n\star \rho_\eta$ to $\hat\psi_n$ in $H_0^1(\Omega)$. This strong convergence also implies (for a subsequence) the quasi-everywhere convergence for the $H_0^1(\Omega)$ capacity and therefore the $\mu$-almost everywhere convergence of $\hat\psi_n\star \rho_\eta$ to $\hat\psi_n$; we use again Lebesgue's dominated convergence Theorem, this time in $L^1(\Omega;d\mu)$, together with the facts that (see \eqref{58bis})
$$
0\leq u^0(\hat\psi_n\star \rho_\eta)\leq u^0 \|\hat\psi_n\|_{L^\infty(\Omega;d\mu)}=u^0 \|\hat\psi_n\|_{L^\infty(\Omega)} \quad \mu\mbox{-a.e. } x\in\Omega
$$
and that (see \eqref{57bis2}) $u^0 \in L^1(\Omega;d\mu)$ in order
to pass to the limit in the second term of the left-hand side. We have proved that
\begin{equation}
\label{subfinalreg*}
\dys\int_{\Omega} A(x)Du^0 D\hat\psi_n+ \int_\Omega u^0 \hat\psi_n d\mu= \int_{\Omega} F(x,u^0)\hat\psi_n.
\end{equation}

We now pass to the limit in each term of \eqref{subfinalreg*} as $n$ tends to infinity. This is easy in the right-hand side by Lebesgue's dominated convergence Theorem since (for a susequence) $\hat\psi_n$ tends almost everywhere to $\psi$, since by the definition of $\hat\psi_n$
$$
0\leq F(x,u^0)\hat\psi_n\leq F(x,u^0)\psi \quad \mbox{ a.e. } x\in\Omega,
$$
and since the latest function belongs to $L^1(\Omega)$ (see \eqref{stimal1}).
This is also easy in the first term of the left-hand side of \eqref{subfinalreg*} since $\hat\psi_n$ tends to $\psi$ strongly in $H_0^1(\Omega)$. 
This strong convergence also implies (for a subsequence) the quasi-everywhere convergence for the $H_0^1(\Omega)$ capacity and therefore the $\mu$-almost everywhere convergence of $\hat\psi_n$ to $\psi$; we use again Lebesgue's dominated convergence Theorem in $L^1(\Omega;d\mu)$,
 together with the facts that (see \eqref{58bis})
$$
0\leq u^0 \hat\psi_n\leq u^0 \psi \leq u^0 \|\psi\|_{L^\infty(\Omega;d\mu)}=u^0 \|\psi\|_{L^\infty(\Omega)}\,\,\, \mu\mbox{-a.e. } x\in\Omega
$$
and that (see \eqref{57bis2}) $u_0 \in L^1(\Omega;d\mu)$ in order
to pass to the limit in the second term of the left-hand side.
We have proved that
\begin{equation}\label{subfinalreg}
\begin{cases}
\forall \psi \in H_0^1(\Omega)\cap L^\infty (\Omega) ,\, \psi \geq 0,\\
\dys\int_{\Omega} A(x)Du^0 D\psi+ \int_\Omega u^0 \psi d\mu= \int_{\Omega} F(x,u^0) \psi.
\end{cases}
\end{equation}

\bigskip
\noindent {\bf Eighth step}

Let us finally prove that $u^0\in L^2(\Omega;d\mu)$ and that \eqref{L_0} holds true.

Taking $\psi=T_n(u^0)\in H_0^1(\Omega)\cap L^\infty(\Omega)$ in \eqref{subfinalreg} we obtain 
\begin{equation*}
\dys\int_{\Omega} A(x)Du^0 DT_n(u^0)+ \int_\Omega u^0 T_n(u^0) d\mu= \int_{\Omega} F(x,u^0) T_n(u^0),
\end{equation*}
in which using the coercitivity \eqref{eq0.0} of $A$ and the growth condition (\ref{eq0.1} {\it iv}) on the function $F$, we obtain
\begin{equation}
\begin{cases}
\dys\into |T_n(u^0)|^2 d\mu\leq \into F(x,u^0)T_n(u^0)\leq \\
\dys\leq \into h(x)\big(\frac{1}{(u^0)^\gamma } + 1 \big)u^0 = 
\into h(x) ((u^0)^{(1 - \gamma)}+ u^0 )<+\infty,
\end{cases}
\end{equation}
which using Fatou's Lemma implies that
$$
u^0\in L^2(\Omega;d\mu).
$$

Fix now  $z\in H_0^1(\Omega)\cap L^2(\Omega;d\mu)$, $z\geq 0$. Taking $\psi = T_n(z)\in H_0^1(\Omega)\cap$ $\cap L^\infty(\Omega)$ as test function in \eqref{subfinalreg} we have
\begin{equation}
\label{921}
\into A(x)Du^0 DT_n(z)+\into u^0 T_n(z)d\mu=\into F(x,u^0)T_n(z).
\end{equation}

It is easy to pass to the limit in each term of the left-hand side of \eqref{921}, since $T_n(z)$ tends to $z$ in $H_0^1(\Omega)\cap L^2(\Omega;d\mu)$ 
and since $u^0\in L^2(\Omega;d\mu)$. 

Applying Fatou's Lemma to the right-hand side of \eqref{921}, we obtain
$$
\into F(x,u^0)z\leq \into A(x)Du^0Dz+\into u^0 zd\mu<+\infty,
$$
which is the first statement of \eqref{L_0}. 

But since $$0\leq F(x,u^0)T_n(z)\leq F(x,u^0)z,$$ and since the latest function belongs to $L^1(\Omega)$, Lebesgue's dominated convergence Theorem implies that
$$
\into F(x,u^0)T_n(z)\rightarrow \into F(x,u^0) z,
$$
which passing to the limit in \eqref{921} completes the proof the second statement of \eqref{L_0}.
\bigskip

The proof of Theorem~\ref{homogenization} is now complete.
\end{proof}

\bigskip

\begin{proof}[{\bf Proof of the corrector Theorem~\ref{corrector}}] $\mbox{}$
\medskip

\noindent {\bf First step}

In view of hypothesis \eqref{517} and of \eqref{cond2}, the function $w^\eps u^0$ belongs to $H_0^1(\Omega)\cap L^\infty(\Omega)$, and therefore the function $r^\eps$ defined by \eqref{518} belongs to $H_0^1(\Omega)$. By the coercivity assumption \eqref{eq0.0} and by the symmetry assumption \eqref{517a} on the matrix $A$, we have
\begin{align}
\label{920bis}
\begin{cases}\vspace{0.1cm}
\dys \alpha \into |D r^\eps|^2\leq \into A(x)Dr^\eps Dr^\eps=\\\vspace{0.1cm}
\dys= \into A(x) (D\tilde{u^\eps}-D(w^\eps u^0)) (D\tilde{u^\eps}-D(w^\eps u^0))=\\\vspace{0.1cm}
\dys = \into A(x) D\tilde{u^\eps} D\tilde{u^\eps} -2\into A(x)D\tilde{u^\eps} D(w^\eps u^0)\,+\\
\dys + \into A(x) D(w^\eps u^0)D(w^\eps u^0).
\end{cases}
\end{align}

We will pass to the limit in each term of the right-hand side of \eqref{920bis}.
\bigskip

\bigskip
\noindent {\bf Second step}

As far as the first term of the right-hand side of \eqref{920bis} is concerned, taking $u^\eps\in H_0^1(\oeps)$ as test function in \eqref{Leps} and extending $u^\eps$ and $F(x,u^\eps)$ by zero into $\tilde{u^\eps}$ and $ \widetilde{F(x,u^\eps)}$ (see \eqref{57bis} and \eqref{defex}), we get

\begin{eqnarray}
\label{c1}
\int_{\Omega}  A(x)D \tilde{u^\eps} D \tilde{u^\eps} = \into  \widetilde{F(x,u^\eps)}\tilde{u^\eps}. 
\end{eqnarray}

Let us pass to the limit in the right-hand side of \eqref{c1}.
By  (\ref{eq0.1} {\it iv})  we have
$$
0\leq F(x,u^\eps)u^\eps\leq h(x)\left(\frac{1}{(u^\eps)^\gamma}+1\right)u^\eps,
$$
which by \eqref{551} and by a  computation similar to the one made in \eqref{836n} implies that for every measurable set $E \subset \Omega$ one has
\begin{align*}
\dys 0\leq \int_E \widetilde{F(x,u^\eps)}\tilde{u^\eps}  
\leq \int_E h(x) \left(\frac{1}{(\tilde{u^\eps})^\gamma}+1\right) \tilde{u^\eps} 
\leq  c \, ( \|h\|_{L^{r}(E)} +  \|h\|_{L^{1}(E)} ),
% \mbox{ for every measurable set } E\subset\Omega,\,\forall\eps,
\end{align*}
where $c$ is a constant which does not depend on $E$ nor on $\eps$, 
which in turn implies  the equi-integrability of $\widetilde{F(x,u^\eps)}\tilde{u^\eps}$.
Then convergences \eqref{ae} and \eqref{good} and Vitali's Theorem imply that
 \begin{equation}
 \label{c2}
  \int_{\Omega}  \widetilde{F(x,u^\eps)}\tilde{u^\eps}  \rightarrow \int_{\Omega} F(x,u^0) u^0.
\end{equation}

On the other hand, taking $z=u^0$ as test function in \eqref{L_0} it follows that
\begin{equation*}
\dys\into  A(x) Du^0 Du^0 + \into (u^0)^2 d\mu=  \int_{\Omega}F(x,u^0)u^0.
\end{equation*}
By \eqref{c1}, \eqref{c2} and the previous equality we have, using \eqref{57bis2} which holds true since $(u^0)^2\in H_0^1(\Omega)$ when $u^0\in L^\infty(\Omega)$,
 \begin{equation}
\label{c3}
\begin{cases}
\dys\int_{\oeps} A(x)D \tilde{u^\eps} D \tilde{u^\eps}\rightarrow \into A(x)  Du^0 D u^0 + \into (u^0)^2 d\mu=\\
\dys= \into A(x)  Du^0 D u^0 + \langle\mu,(u^0)^2\rangle_{H^{-1}(\Omega),H_0^1(\Omega)}.
\end{cases}
\end{equation}

\bigskip
\noindent {\bf Third step}

Let us now pass to the limit in the third term of the right-hand side of \eqref{920bis}. Using \eqref{cond5} we obtain 
\begin{align*}
\begin{cases}\vspace{0.1cm}
\dys\into A(x)D(w^\eps u^0)D(w^\eps u^0)=\\\vspace{0.1cm}
\dys =\into A(x)D(w^\eps u^0) Dw^\eps u^0+\into A(x) D(w^\eps u^0) Du^0w^\eps=\\\vspace{0.1cm}
\dys = \into {}^t\!A(x)Dw^\eps D(w^\eps (u^0)^2) -\into {}^t\!A(x)Dw^\eps Du^0 w^\eps u^0 \, + \\\vspace{0.1cm}
\dys +\into A(x)D(w^\eps u^0) Du^0 w^\eps=\\\vspace{0.1cm}
\dys= \langle\mu^\eps,w^\eps(u^0)^2\rangle_{H^{-1}(\Omega),H_0^1(\Omega)} -\into {}^t\!A(x)Dw^\eps Du^0 w^\eps u^0 \,+\\
\dys +\into A(x) D(w^\eps u^0) Du^0w^\eps,
\end{cases}
\end{align*}
in which it is easy to pass to the limit in each term, obtaining 
\begin{equation}
\label{942}
\begin{cases}
\dys \into A(x)D(w^\eps u^0)D(w^\eps u^0)\rightarrow\\
\dys \rightarrow\langle \mu,(u^0)^2\rangle_{H^{-1}(\Omega),H_0^1(\Omega)}+\into A(x) Du^0 Du^0.
\end{cases}
\end{equation}

\bigskip
\noindent {\bf Fourth step}

Passing to the limit in the second term of the right-hand side of \eqref{920bis} is a little bit more delicate 
(except in the case where the regularity hypothesis \eqref{nuch} is made on the function $F$, see Remark~\ref{rem9.1} below). 

Fix $\phi \in \mathcal{D}(\Omega)$ and write
\begin{equation}
\label{943}
\begin{cases}
\dys\into A(x) D\tilde{u^\eps} D(w^\eps u^0)=\into A(x)D\tilde{u^\eps}Du^0w^\eps+\\
\dys +\into A(x) D\tilde{u^\eps}D w^\eps\phi+\into A(x)D\tilde{u^\eps}Dw^\eps (u^0-\phi).
\end{cases}
\end{equation}

It is easy to pass to the limit in the first term of the right-hand side of \eqref{943}, obtaining 
\begin{equation}
\label{944}
\into A(x) D\tilde{u^\eps}Du^0 w^\eps\rightarrow \into A(x)Du^0Du^0.
\end{equation}

For what concerned the second term of the right-hand side of \eqref{943}, we have in view of \eqref{cond5}
\begin{align}
\label{944bis}
\begin{cases}
\dys\into A(x)D\tilde{u^\eps}Dw^\eps\phi=\into {}^t\!A(x)Dw^\eps D \tilde{u^\eps} \phi=\\
\dys = \into {}^t\!A(x)Dw^\eps D(\tilde{u^\eps}\phi)-\into {}^t\!A(x)Dw^\eps D\phi\, \tilde{u^\eps}=\\
\dys = \langle\mu^\eps,\tilde{u^\eps}\phi\rangle_{H^{-1}(\Omega),H_0^1(\Omega)}-\into {}^t\!A(x)Dw^\eps D\phi\, \tilde{u^\eps},
\end{cases}
\end{align} 
and therefore
\begin{equation}
\label{945}
\begin{cases}
\dys\into A(x) D\tilde{u^\eps}Dw^\eps\phi\rightarrow \langle \mu,u^0\phi\rangle_{H^{-1}(\Omega),H_0^1(\Omega)}=\\
\dys =\langle \mu,(u^0)^2\rangle_{H^{-1}(\Omega),H_0^1(\Omega)}+\langle \mu,u^0(\phi-u^0)\rangle_{H^{-1}(\Omega),H_0^1(\Omega)}.
\end{cases}
\end{equation}

\bigskip
\noindent {\bf Fifth step}

We now use $w^\eps(u^0-\phi)^2$ as test function in \eqref{cond5}. This gives
\begin{align*}
\begin{cases}
 \dys\into {}^t\!A(x)Dw^\eps Dw^\eps (u^0-\phi)^2+2 \into {}^t\!A(x)Dw^\eps D(u^0-\phi)(u^0-\phi)w^\eps=\\
 \dys =  \langle\mu^\eps,w^\eps(u^0-\phi)^2\rangle_{H^{-1}(\Omega),H_0^1(\Omega)},
\end{cases}
\end{align*}
which implies that
$$
 \into {}^t\!A(x)Dw^\eps Dw^\eps (u^0-\phi)^2\rightarrow \langle\mu,(u^0-\phi)^2\rangle_{H^{-1}(\Omega),H_0^1(\Omega)}.
$$
By the coercivity \eqref{eq0.0}, this implies that, for every $\phi\in\mathcal{D}(\Omega)$,
$$
\limsup_{\eps}\, \alpha \into|Dw^\eps|^2  |u^0-\phi|^2\leq \langle\mu,(u^0-\phi)^2\rangle_{H^{-1}(\Omega),H_0^1(\Omega)}.
$$

This result together with H\"older's inequality and 
the bound \eqref{551} on $\|\tilde{u^\eps}\|_{H_0^1(\Omega)}$ implies that for every $\phi\in\mathcal{D}(\Omega)$ 
\begin{align}
\label{946}
\begin{cases}
\dys\limsup_{\eps}\left|\into A(x)D\tilde{u^\eps}Dw^\eps(u^0-\phi)\right|\leq\\
\dys\leq \|A\|_{L^\infty(\Omega)^{N\times N}}\|\tilde{u^\eps}\|_{H_0^1(\Omega)}
\left( \limsup_{\eps} \into|Dw^\eps|^2  |u^0-\phi|^2 \right)^{1/2} \leq \\
\dys\leq \|A\|_{L^\infty(\Omega)^{N\times N}} \, C(|\Omega|,N,\alpha,\gamma, r)\,  \big( \|h\|_{L^{r}(\Omega)} + \|h\|_{L^{1}(\Omega)}^{1/2} \big) \\
\dys \hskip 4.5cm \left( \frac{1}{\alpha} \langle\mu,(u^0-\phi)^2\rangle_{H^{-1}(\Omega),H_0^1(\Omega)} \right)^{1/2} \leq \\
%C(|\Omega|,N,\alpha,\gamma, r) \left( \|h\|_{L^{r}(\Omega)} + \|h\|_{L^{1}(\Omega)}^{1/2} \right)
\dys \leq c \left(\langle\mu,(u^0-\phi)^2\rangle_{H^{-1}(\Omega),H_0^1(\Omega)}\right)^{1/2},
\end{cases}
\end{align}
where $C(|\Omega|,N,\alpha,\gamma, r)$ is the constant which appears in \eqref{num11bis}, 
and where the constant $c$ depends only on 
$|\Omega|$, $N$, $\alpha$, $ \|A\|_{L^\infty(\Omega)^{N\times N}}$, $ \gamma$, $ r$, $ \|h\|_{L^{r}(\Omega)}$ and $\|h\|_{L^{1}(\Omega)}$.

%fm%
%Collecting the results obtained in \eqref{943}, \eqref{944}, \eqref{945} and \eqref{946}, we have proved that for every $\phi\in \mathcal{D}(\Omega)$ one has 
%\begin{align}
%\label{947}
%\begin{cases}\vspace{0.1cm}
%\dys\limsup_{\eps} \left[\into A(x)D\tilde{u^\eps}D(w^\eps u^0)\,+\right.\\\vspace{0.1cm}
%\hskip 1.5cm\dys\left.-\into A(x)Du^0Du^0- \langle\mu,(u^0)^2\rangle_{H^{-1}(\Omega),H_0^1(\Omega)}\right]\leq\\
%\dys \leq \left| \langle\mu,u^0(\phi-u^0)\rangle_{H^{-1}(\Omega),H_0^1(\Omega)}\right|+c \langle\mu,(u^0-\phi)^2\rangle_{H^{-1}(\Omega),H_0^1(\Omega)},
%\end{cases}
%\end{align}
%where $c$ does not depend on $\phi$.
%\bigskip

%
\bigskip
\noindent {\bf Sixth step}

Using in \eqref{920bis} the results obtained in \eqref{c3}, \eqref{942}, \eqref{943}, \eqref{944}, \eqref{945} and \eqref{946}, we have proved that for every $\phi\in\mathcal{D}(\Omega)$ one has 
\begin{align}
\label{948}
\begin{cases}
\dys\limsup_{\eps} \, \alpha \, \|r^\eps\|^2_{H_0^1(\Omega)}\leq \\
\dys \leq -2 \langle\mu,u^0(\phi-u^0)\rangle_{H^{-1}(\Omega),H_0^1(\Omega)}
+ \limsup_{\eps} \left( -2 \into A(x)D\tilde{u^\eps}Dw^\eps(u^0-\phi) \right) \leq \\
\dys \leq -2 \langle\mu,u^0(\phi-u^0)\rangle_{H^{-1}(\Omega),H_0^1(\Omega)} 
+ 2c \left( \langle\mu,(u^0-\phi)^2\rangle_{H^{-1}(\Omega),H_0^1(\Omega)} \right)^{1/2}.
\end{cases}
\end{align}

Since the sequence $z_n^2$ converges to $0$ strongly in $H_0^1(\Omega)$ when $z_n$ converges to $0$ strongly in $H_0^1(\Omega)$ and weakly-star in $L^\infty(\Omega)$,
approximating $u^0\in H_0^1(\Omega)\cap L^\infty(\Omega)$ by a sequence of functions $\phi\in\mathcal{D}(\Omega)$ which converges to $u^0$ strongly in $H_0^1(\Omega)$ and weakly-star in $L^\infty(\Omega)$ proves that 
$$
r^\eps\rightarrow 0 \mbox{ in } H_0^1(\Omega) \mbox{ strongly},
$$
i.e. \eqref{518}. 
\bigskip

Theorem~\ref{corrector} is proved.
\end{proof}

\begin{remark}\label{rem9.1}
The above proof of Theorem~\ref{corrector} has been made  assuming that  \eqref{517} holds true, namely that $u^0 \in L^\infty (\Omega)$. 
If we assume that  the function~$F$, in addition to assumption \eqref{car} and \eqref{eq0.1}, verifies the regularity condition \eqref{nuch}, the proof of the corrector 
Theorem~\ref{corrector} becomes simpler.

Indeed under this hypothesis, the solutions $\tilde u^\eps$ are bounded in $L^\infty (\Omega)$ (see Remark~\ref{rem55}) (a result which, by the way, implies \eqref{517}).
We claim that this $L^\infty (\Omega)$ bound on $\tilde u^\eps$ allows us to perform the computation in the fourth step above when we replace $\phi$ by $u^0$: 
indeed in this case the third term of the right-hand side of \eqref{943} vanishes
(and therefore the fifth step  becomes useless), and \eqref{944bis} becomes
\begin{align}
\label{944ter}
\begin{cases} \vspace{0.1cm}
\dys\into A(x)D\tilde{u^\eps}Dw^\eps \, u^0=\into {}^t\!A(x)Dw^\eps D \tilde{u^\eps} \,u^0=\\ \vspace{0.1cm}
\dys = \into {}^t\!A(x)Dw^\eps D(\tilde{u^\eps} u^0)-\into {}^t\!A(x)Dw^\eps Du^0 \, \tilde{u^\eps}=\\ \vspace{0.1cm}
\dys = \langle\mu^\eps,\tilde{u^\eps} u^0\rangle_{H^{-1}(\Omega),H_0^1(\Omega)}-\into {}^t\!A(x)Dw^\eps Du^0\, \tilde{u^\eps},
\end{cases}
\end{align}
where each term has a meaning since now $\tilde u^\eps \in L^\infty (\Omega)$. 
Using the fact that $\tilde u^\eps$ is bounded in $L^\infty (\Omega)$ allows us to pass to the limit in \eqref{944ter}, obtaining
$$
\dys\into A(x) D\tilde{u^\eps}Dw^\eps \, u^0 \rightarrow \langle \mu, (u^0)^2\rangle_{H^{-1}(\Omega),H_0^1(\Omega)}.
$$
Then \eqref{948} becomes
$$
\limsup_{\eps} \, \alpha\,  \|r^\eps\|^2_{H_0^1(\Omega)}\leq 0,
$$
which is nothing but the desired result \eqref{518}. 
\qed
\end{remark}

\vskip2cm

\bigskip
 \noindent \textbf{Acknowledgments}. The authors would like to thank their institutions
 % 24 fev 
 %(Sapienza Universit\`a di Roma, Universit\'e Pierre et Marie Curie Paris VI, Universidad Polit\'ecnica de Cartagena) 
 for providing the support of reciprocal visits which allowed them to perform the present work. 
%fm% The work of Pedro J. Mart\'inez-Aparicio was partially supported by  MINECO grant MTM2015-68210-P (Spain),  grant FQM-116  of the Spanish Junta de Andaluc\'ia and grant Programa de Apoyo a la Investigaci\'on de la Fundaci\'on S\'eneca-Agencia de Ciencia y Tecnolog\'ia de la Regi\'on de Murcia, 19461/PI/14. 
The work of Pedro J. Mart\'inez-Aparicio has been partially supported by the MINECO grant MTM2015-68210-P, by the grant FQM-116 of the Junta de Andaluc\'ia,  and by the Programa de Apoyo a la Investigaci\'on 19461/PI/14 of the Fundaci\'on S\'eneca-Agencia de Ciencia y Tecnolog\'{\i}a de la Regi\'on de Murcia.

\end{document}